\RequirePackage{fix-cm}
\documentclass[final]{svjour3}                     
\smartqed  
\usepackage{graphicx}
\usepackage{amsmath}%
\usepackage{amstext}%
\usepackage{array}
\usepackage{multirow}
\newcolumntype{C}[1]{>{\centering\arraybackslash}p{#1}}
\newcolumntype{P}[1]{>{\centering\arraybackslash}p{#1}}
\newcolumntype{M}[1]{>{\centering\arraybackslash}m{#1}}
\usepackage{amssymb}%
\usepackage{amsxtra}
\usepackage{epsfig}%
\usepackage{cite}
\usepackage{tikz}
\usepackage{caption}
\usepackage{float}
\usepackage{longtable}
\usepackage{threeparttable,booktabs,multirow,lscape}
\usepackage{lipsum}
\usepackage{rotating}
\usepackage{multirow}
\usepackage{booktabs}
\usepackage{amsmath}
\usepackage{amssymb}
\usepackage{graphicx}
\usepackage[colorlinks,urlcolor=red]{hyperref}
\usepackage{graphicx}
\usepackage{subfigure}
\usepackage{listings}
\usepackage[ruled,vlined]{algorithm2e}

\newtheorem{fact}[theorem]{Fact}

\newcommand{\R}{\mathbb{R}}%
\newcommand{\N}{\mathbb{N}}%
\newcommand{\Hi}{\mathcal{H}}%

\renewcommand{\>}{\right\rangle}
\newcommand{\<}{\left\langle}

\DeclareMathOperator*\fix{Fix}

\begin{document}

\title{Extrapolated Sequential Constraint Method for Variational Inequality over the Intersection of Fixed-Point Sets
}

\titlerunning{Extrapolated Sequential Constraint Method}        

\author{Mootta Prangprakhon        \and
        Nimit Nimana 
}

\authorrunning{M. Prangprakhon and N. Nimana} 

\institute{M. Prangprakhon  \at
              Department of Mathematics, Faculty of Science, Khon Kaen University, Khon Kaen, 40002 Thailand,
              \email{mootta\_prangprakhon@hotmail.com}        
           \and
           N. Nimana\at
              Department of Mathematics, Faculty of Science, Khon Kaen University, Khon Kaen, 40002 Thailand,
               \email{nimitni@kku.ac.th}  
}

\date{Received: date / Accepted: date}

\maketitle

\begin{abstract}
This paper deals with the solving of variational inequality problem where the constrained set is given as the intersection of a number of fixed-point sets. To this end, we present  an extrapolated sequential constraint method. At each iteration, the proposed method is updated based on the ideas of a hybrid conjugate gradient method used to accelerate the well-known hybrid steepest descent method, and an extrapolated cyclic cutter method for solving a common fixed point problem. We prove strong convergence of the method under some suitable assumptions of step-size sequences. We finally show the numerical efficiency of the proposed method compared to some existing methods.
\keywords{Conjugate gradient direction \and Cutter \and Fixed point \and Hybrid steepest descent method \and Variational inequality}
\end{abstract}

\section{Introduction}
\label{intro}

In this paper, we consider the following variational inequality problem:
\begin{problem}\label{Problem-VIP} 
	Let $T_i:\Hi\to\Hi$,  $i=1,2,\ldots,m$, be cutters with $\bigcap\limits_{i = 1}^m \fix{{T_i}}\ne \emptyset$, and let $F:\Hi\to\Hi$ be $\eta$-strongly monotone and $\kappa$-Lipschitz continuous. Then, our objective is to find a point $\bar u\in \bigcap\limits_{i = 1}^m \fix{{T_i}}$ such that 
	\[\langle F({\bar u}),z- {\bar u}\rangle  \ge 0 \indent\text{    for all $z\in \bigcap\limits_{i = 1}^m \fix{{T_i}}$.   }\]
\end{problem}
 Attentively, Problem \ref{Problem-VIP} has a bilevel structure, namely, its outer level given by the variational inequality govern by the operator $F$, while the constrained set is the inner level problem, which is the common fixed point problem of cutter operators. We emphasize here the importance of Problem \ref{Problem-VIP} is not only the allowing us a generalization of the constrained set, but also various applications for modelling real-world problems like network location problems \cite{I13,I15-2,IH14,I19-2}, and machine learning \cite{I19-3}, to name but a few.

 For simplicity, we denote by VIP($F,C$) a variational inequality problem corresponding to an operator $F$ and a nonempty closed convex set $C$. In the literature, the simplest iterative algorithm for solving VIP($F,C$)  is the well-known {\it projected gradient method} (PGM) \cite{G64}. The method essentially has the form:
\begin{equation}\label{PGM}
\left \{
\begin{aligned}
&x^1\in C \text{       is arbitrarily chosen,       }  \\
&x^{n+1}=\mathrm{proj}_{C}(x^n-\mu F(x^n)),
\end{aligned} \right.
\end{equation} 
for every $n\in\mathbb{N}$, where $\mathrm{proj}_{C}:\Hi\to C$ is the metric projection onto $C$, $F:\Hi\to C$ is $\eta$-strongly monotone and $\kappa$-Lipschitz continuous over $C$ and  $\mu\in(0,2\eta/\kappa^2)$. It was proved that the sequence $\{x^n\}_{n=1}^\infty$ generated by (\ref{PGM}) converges strongly to the unique solution of VIP($F,C$) in \cite{G64}.
As PGM requires the use of the metric projection $\mathrm{proj}_{C}$, it is perfectly suitable for the case when $C$ is simple enough in the sense that $\mathrm{proj}_{C}$ has a closed-form expression. However, in many practical situations, the structure of $C$ can be highly intricate and, in consequence, $\mathrm{proj}_{C}$ is difficult to evaluate.
To overcome the above limitation, Yamada \cite{Y01} proposed the celebrated {\it hybrid steepest descent method} (HSDM) which essentially replaces the use of $\mathrm{proj}_{C}$ in (\ref{PGM}) with an appropriate nonexpansive operator $T$. By intepreting $C$ as the fixed point set of $T$, the method is defined by the following:
\begin{equation}\label{HSDM}
\left \{
\begin{aligned}
&x^1\in \Hi \text{       is arbitrarily chosen,       }  \\
&x^{n+1}=T(x^n-\mu\beta_nF(x^n)),
\end{aligned} \right.
\end{equation} 
for every $n\in\mathbb{N}$, where $F:\Hi\to\Hi$ is $\eta$-strongly monotone and $\kappa$-Lipschitz continuous over $\Hi$, and $\mu\in(0, 2\eta/\kappa^2)$. It is well-known that, under some certain conditions on $\{\beta_n\}_{n=1}^\infty\subset(0,1]$, the sequence $\{x^n\}_{n=1}^\infty$ generated by (\ref{HSDM}) converges strongly to the unique solution of  VIP($F,\fix T$), where $\fix T := \{ x \in \Hi: Tx = x\}$. 
Note that, in the context of (\ref{HSDM}), if $F:=\nabla f$ where $f:\Hi\to\mathbb{R}$ is a convex, continuously Fr\'echet differentiable functional, HSDM thus solves VIP($\nabla f,\fix T$), which is nothing else than the convex minimization problem over the fixed point set of a nonexpansive operator. On the other hand, it is well-known that the {\it conjugate gradient method} (CGM)\cite{NW99, DY99, FR64,GN92} 
and the {\it three-term conjugate gradient method} (TCGM)\cite{ZZLI06-1,ZZLI06-2,ZZLI06-3} 
have great efficacy in decreasing the function $f$ value rapidly. According to these underline motivations, several modifications among HSDM, CGM and TCGM are proposed  in order to accelerate HSDM, namely, the {\it hybrid conjugate gradient method} (HCGM) \cite{IY09}, the {\it hybrid three-term conjugate gradient method} (HTCGM)\cite{I11} and the {\it accelerated hybrid conjugate gradient method} (AHCGM) \cite{I15}.
As a matter of fact, HCGM and HTCGM are relatively similar in some basic structures and some additional conditions needed to ensure their convergences. In addressing such procedures, their common form is as follows:
\begin{equation}\label{HCGM}
\left \{
\begin{aligned}
&x^1\in\Hi \text{       is arbitrarily chosen,       }  \\
&d^1=-\nabla f(x^1),\\
& x^{n+1}=T(x^n+\mu\beta_nd^n),\\
\end{aligned} \right.
\end{equation} 
for every $n\in\mathbb{N}$, where $\mu\in(0, 2\eta/\kappa^2)$, $\{\beta_n\}_{n=1}^\infty\subset(0,1]$ is a step size and $\{d^n\}_{n=1}^\infty\in\Hi$ is a search direction. However, it is worth mentioning that the search directions of these methods are slightly different, that is, the search direction of HCGM is defined by
\begin{equation}\label{CGM}
d^{n}=-\nabla f(x^{n})+\varphi_{n}^{(1)}d^{n-1},
\end{equation}
meanwhile the search direction of HTCGM is defined by
\begin{equation}\label{TCGM}
d^{n}=-\nabla f(x^{n})+\varphi_{n}^{(1)}d^{n-1}-\varphi_{n}^{(2)}w^n,
\end{equation}
for every $n\in\mathbb{N}$, where $\{\varphi_n^{(i)}\}_{n=1}^\infty\subset[0,\infty)(i=1,2)$ and $\{w^n\}_{n=1}^\infty\in\Hi$ is arbitrarily chosen. Then, it was proved in \cite{IY09} and \cite{I11}  that, under some certain assumptions on $\{\beta_n\}_{n=1}^\infty\subset(0,1]$, each sequence generated by HCGM and HTCGM converges strongly to the unique solution of VIP($\nabla f,\fix T$) whenever $ \mathop {\lim }\limits_{n \to \infty }{\varphi_n^{(i)}} = 0 (i=1,2)$, and the sequences $\{\nabla f(x^n)\}_{n=1}^\infty$ and $\{w^n\}_{n=1}^\infty$ are bounded.

Next, let us review some sequential methods used for solving the common fixed point problem (in short, CFPP). Namely,  let $T_i:\Hi\to\Hi$, $i=1,2,\ldots,m$, be nonlinear operators, the problem is to find $$x^*\in\bigcap\limits_{i = 1}^m \fix{T_i},$$ provided that the intersection is nonempty. A classical sequential method for solving CFPP was developed from an iterative method introduced by Kaczmarz \cite{K37} who firstly aimed to solve a linear system in $\mathbb{R}^n$. The method was referred to the {\it cyclic projection method} (CPM) or {\it Kaczmarz method} (KM) which has the form:
\begin{equation}\label{CPM}
\left \{
\begin{aligned}
&x^1\in\Hi \text{       is arbitrarily chosen,       }  \\
& x^{n+1}=\mathrm{proj}_{C_m}\cdots \mathrm{proj}_{C_1}x^n,\\
\end{aligned} \right.
\end{equation} 
where $\mathrm{proj}_{C_i}$ are the metric projections onto the linear equations $C_i\subset\Hi$, $i=1,2,\ldots,m$. After that, the general case when $C_i\subset\Hi$, $i=1,2,\ldots,m$, are nonempty closed and convex subsets was considered by Bregman \cite{B65}. It was  proved that the sequence generated by (\ref{CPM}) converges weakly to a solution of  CFPP. 
As the interest in the aforementioned results continuously increase, it is well-known that, under some additional hypotheses, the convergence of CPM is true for a wider class of operators such as nonexpansive operators or cutter operators \cite{O67,C12,CC11,L95,C10}. In particular, the latter is a key tool of a method called the {\it cyclic cutter method} (CCM) which its weak convergence was proved by Bauschke and Combettes \cite{BC01}.
In order to accelerate the convergence of CCM, Cegielski and Censor \cite{CC12} proposed the so-called {\it extrapolated cyclic cutter method} (ECCM) which essentially requires the use of an appropriate step-size function $\sigma:\Hi\to(0,\infty)$ to speed up numerically  the convergence behaviour. Indeed, let $T_i:\Hi\to\Hi$, $i=1,2,\ldots,m$, be cutters with $\bigcap\limits_{i = 1}^m {\fix{T_i}\ne \emptyset}$, define $T: = {T_m}{T_{m - 1}}\cdots{T_1}$, $S_0:=Id$ and  $S_i:=T_iT_{i-1}\cdots T_1$, then they defined the step-size function $\sigma$ as
\begin{equation}\label{sigma}
\sigma (x):=\left\{ 
\begin{array}{ll}
\displaystyle\frac{\sum_{i=1}^{m}\langle Tx-S_{i-1}x,S_{i}x-S_{i-1}x\rangle 
}{\Vert Tx-x\Vert ^{2}}, & \text{for \ }x\notin \bigcap\limits_{i = 1}^m \fix{T_i}, \\ 
1, & \text{otherwise.}%
\end{array}%
\right. 
\end{equation}%
Moreover, it was shown that ECCM converges weakly whenever the cutter operators $T_i$, $i=1,2,\ldots,m$, satisfy the demi-closedness principle. 
Along the line of \cite{CC12}, Cegielski and Nimana \cite{CN19} indicated that there are some practical situations in which the value of the extrapolation function $\sigma$ can be enormously large, which consequently may produce some uncertainties in numerical experiments. In order to avoid these situations, they proposed an algorithm called the {\it modified extrapolated cyclic subgradient projection method} (MECSPM). The main idea of this method is to map each iterate obtaining from ECCM via the last subgradient projection. If the constrained sets are nonempty closed convex sets, the modification is nothing else than the projecting a sequence generated by ECCM into the last constraint set.
To conclude, the aforementioned methods used for solving variational inequality problem and common fixed point problem are concisely summarized in Table \ref{tab:1}.
\begin{table}[h]
	\caption{Summary of the corresponding iterative methods used for solving Problem \ref{Problem-VIP}.}
	\label{tab:1}       
	\begin{tabular}{l  c c c c  c }
		\hline\noalign{\smallskip}
		\textbf{Reference}& \textbf{Problem} & \textbf{Method} & \textbf{Constrained Operator}  \\
		\noalign{\smallskip}\hline\noalign{\smallskip}
	Goldstein \cite{G64} & VIP$(F, C)$ & PGM & metric projection \\ 
	Yamada \cite{Y01} & VIP$(F, \fix T)$ & HSDM & nonexpansive  \\ 
	Iiduka $\&$ Yamada \cite{IY09}& VIP$(\nabla f, \fix T)$ & HCGM & nonexpansive  \\ 
	Iiduka \cite{I11} & VIP$(\nabla f, \fix T)$ & HTCGM & nonexpansive \\ 
	Bregman \cite{B65} & CFPP & CPM & metric projection \\ 
	Bauschke $\&$ Combettes \cite{BC01}  & CFPP & CCM & cutter \\ 
	Cegielski $\&$ Censor \cite{CC12} & CFPP & ECCM &  cutter\\ 
	Cegielski $\&$ Nimana \cite{CN19}  & CFPP & MECSPM &  subgradient projection \\ 
	{\bf This work} & VIP$\left(F, \bigcap\limits_{i = 1}^m \fix{{T_i}}\right)$& ESCoM-CGD & cutter \\
		\noalign{\smallskip}\hline
	\end{tabular}
\end{table}

The main contribution of this paper is an iterative algorithm called the {\it extrapolated sequential constraint method with conjugate gradient direction} (ESCoM-CGD) used for solving the variational inequality problem over the intersection of the fixed-point sets. To construct the algorithm, we utilize some ideas of the aforementioned methods, namely, HCGM \cite{IY09} and MECSPM\cite{CN19}. 
Under the context of cutter operators and some certain conditions, we establish strong convergence of the proposed algorithm. In order to demonstrate the effectiveness and the performance of the algorithm, we present numerical results and numerical comparisons of the algorithm with some existing methods such as HCGM and HTCGM. 

The remainder of this paper is organized as follows. In Section 2, we  collect some useful definitions and results needed in the paper. In Section 3, we introduce ESCoM-CGD used for solving Problem \ref{Problem-VIP} and subsequently analyse its convergence result. In Section 4, we derive an important situation of the considered problem by means of the subgradient projection. In Section 5, the efficacy of ESCoM-CGD is illustrated by some numerical results. Finally, we give some concluding remarks in Section 6.

\section{Preliminaries}
Throughout the paper, $\Hi$ is always a real Hilbert space with an inner product $\langle \cdot , \cdot \rangle$ and with the norm $\parallel  \cdot \parallel$. For a sequence $\{ {x^n}\} _{n = 1}^\infty$, the expressions ${x^n}\rightharpoonup x$ and ${x^n} \to x$ denote $\{ {x^n}\} _{n = 1}^\infty$ converges to $x$ weakly and converges to $x$ in norm, respectively. $Id$ represents the identity operator on $\Hi.$ 


An operator $F:\Hi\to\Hi$ is said to be
$\eta$-{\it strongly monotone} if  there exits a constant $\eta>0$ such that
$\langle Fx-Fy,x-y \rangle \ge \eta \|x-y\|^2,$
for all $x,y\in\Hi$, and is said to be $\kappa$-{\it Lipschitz continuous} if  there exits a constant $\kappa>0$ such that
$\|Fx-Fy\|\le \kappa\|x-y\|,$
for all $x,y\in\Hi$.

The following lemma found in \cite[Lemma 3.1(b)]{Y01} will be useful in the sequel. 
\begin{lemma}\label{yamada} Suppose that $F:\Hi\to\Hi$ is $\eta$-strongly monotone and $\kappa$-Lipschitz continuous. For any $\mu  \in (0,2\eta /{\kappa ^2})$ and $\beta  \in (0,1]$, define the operator $T^\beta:\Hi\to\Hi$ by 
	${T^\beta }: = Id -  \mu \beta F$. Then 
	\[\|{T^\beta }x - {T^\beta }y\| \le (1 - \beta \tau )\|x - y\|,\]
	for all $x,y\in\Hi$, where $\tau : = 1 - \sqrt {1 + {\mu ^2}{\kappa ^2} - 2\mu \eta }  \in (0,1].$
\end{lemma}
\begin{remark}
	It is worth to notice that the well definedness of the parameter $\tau\in(0,1]$ is guaranteed by the assumption of $F$. Indeed, the monotonicity of $F$ and the Cauchy-Schwarz inequality yield that
	$ \eta \|x - y{\|^2}\le\langle F(x) - F(y),x - y\rangle\le\|F(x)-F(y)\|\|x-y\|,$
	and hence
	$\eta\|x-y\|\le\|F(x)-F(y)\|.$
	Due to the Lipschitz continuity of $F$, we obtain
	$\|F(x)-F(y)\|\le\kappa\|x-y\|,$
	which implies that
	$$0<\eta\le\kappa.$$ 
	Thus, we have $0<\frac{2\eta}{\kappa^2}$. Setting $\mu\in(0,\frac{2\eta}{\kappa^2})$, we obtain $$0 \le {(1 - \mu \kappa )^2} \le 1 + {\mu ^2}{\kappa ^2} - 2\mu \eta  < 1.$$ 
	Therefore $$0 < 1 - \sqrt {1 + {\mu ^2}{\kappa ^2} - 2\mu \eta }  \le 1,$$
	which means that $\tau\in(0,1]$.
\end{remark}
Below, some concepts of quasi-nonexpansivity of operators are presented for the sake of further use. More details can be found in \cite[Section 2.1.3]{C12}.

An operator $T:\Hi\to\Hi$ with $\fix T\ne\emptyset$ is said to be
{\it quasi-nonexpansive} if 
$\|Tx-z\|\le\|x-z\|,$
for all $x\in\Hi$ and for all $z\in\fix T$,
is said to be {\it $\rho$-strongly quasi-nonexpansive}, where $\rho\ge0$, if  
$\|Tx - z{\|^2} \le \|x - z{\|^2} - \rho \|Tx - x{\|^2},$
for all $x\in\Hi$ and for all $z\in\fix T$,
and, is said to be a {\it cutter }if 
$ \langle x - Tx,z - Tx\rangle  \le 0,$
for all $x\in\Hi$ and for all $z\in\fix T$.
\begin{fact}\label{fact1}
	If $T:\Hi\to\Hi$ is quasi-nonexpansive, then $\fix T$ is closed and convex.
\end{fact}

\begin{fact}\label{fact2}
	Let $T:\Hi\to\Hi$ be a cutter. Then the following properties hold:
	
	(i) $\langle Tx - x,z - x\rangle  \ge \|Tx - x{\|^2}$ for every $x\in\Hi$ and $z\in\fix T$.
	
	(ii) $T$ is 1-strongly quasi-nonexpansive.
	
	
\end{fact}

We recall a notion of the demi-closedness principle in the following definition.
\begin{definition}
	An operator $T:\Hi \to \Hi$ is said to satisfy the {\it demi-closedness} (DC) principle if $T-Id$ is demi-closed at $0$, that is, for any sequence $\{x^n\}_{n=1}^\infty\subset\Hi$, if $x^n\rightharpoonup y\in\Hi$ and $\|(T-Id)x^n\| \to 0$, then $Ty=y$.
\end{definition}

Further, we recall that an operator $T:\Hi\to\Hi$ is said to be {\it nonexpansive} if 
$\|Tx-Ty\|\le\|x-y\|,$ for all $x,y\in\Hi$. 
It is worth mentioning that if $T:\Hi\to\Hi$ is a nonexpansive operator with $\fix T\neq\emptyset$, then the operator $T$ satisfies the DC principle (see \cite[Lemma 2]{Z71}).

For an operator $T:\Hi\to\Hi$ and a real number $\lambda\in[0,2]$, the operator $T_\lambda:=(1-\lambda)Id+\lambda T$ is called a {\it relaxation} of $T$ and $\lambda$ is called a relaxation parameter. Actually, in many situations, the relaxation parameter which is greater than $2$ may yield a superiority of algorithmic convergence property. So, we are now in a position to recall a generalized relaxation of an operator.
The generalized relaxation of an operator $T:\Hi\to\Hi$ is defined by 
${T_{\sigma ,\lambda }}x: = x + \lambda \sigma (x)(Tx - x),$
where $\sigma:\Hi\to(0,\infty)$ is a step-size function.
If  $\sigma (x) \ge 1$ for all $x\in\Hi$, then the operator ${T_{\sigma ,\lambda }}$ is called an {\it extrapolation} of $T$. In the case that $\sigma(x)=1$, for all $x\in\Hi$, the generalized relaxation of $T$ is reduced to the relaxation of $T$, that is $T_{\sigma,\lambda}=T_\lambda$. We denote here that $T_{\sigma}:=T_{\sigma,1}$. For any $x\in\Hi$, it can be noted that 
$$T_{\sigma,\lambda}x-x=\lambda\sigma(x)(Tx-x)=\lambda(T_\sigma x-x),$$ i.e., $T_{\sigma,\lambda}x=x+\lambda(T_\sigma x-x),$ and $$\fix T_{\sigma,\lambda}=\fix T_\sigma=\fix T,$$ for any $\lambda\ne0$.

The following lemma plays an important role in proving our convergence result. The proof can be found in \cite[Section 4.10]{C12}.
\begin{lemma}\label{lemma-CC12}
	Let ${T_i}:\Hi \to \Hi, i = 1,2,\ldots,m,$ be cutters with $\bigcap\limits_{i = 1}^m {\fix{T_i}}  \ne \emptyset$, and denote $T:=T_mT_{m-1}\cdots T_1$. Let $\sigma:\Hi\to(0,\infty)$ be defined by (\ref{sigma}), then the following properties hold:
	\begin{itemize}
		\item[(i)] For any $x\notin\fix T$, we have
		$${\sigma }(x) \ge \frac{{\frac{1}{2}\sum\limits_{i = 1}^m {\|{S_i}x - {S_{i - 1}}x{\|^2}} }}{{\|Tx - x{\|^2}}} \ge \frac{1}{{2m}},$$
		where $S_0=Id$ and  $S_i=T_iT_{i-1}\cdots T_1$.
		\item[(ii)] The operator $T_\sigma$ is a cutter.
	\end{itemize}
\end{lemma}

\section{Algorithms and Convergence Results}
In this section, we start with the introducing a new iterative algorithm for solving Problem \ref{Problem-VIP} and subsequently study its convergence result. For the sake of convenience, we denote the following notations: the compositions 
$T := {T_m}{T_{m - 1}}\cdots {T_1},$
$S_0:=Id,$ and $S_i:=T_iT_{i-1}\cdots T_1$, $i=1,2,\dots,m,$
where $T_i:\Hi\to\Hi$,  $i=1,2,\ldots,m$, are cutters with $\bigcap\limits_{i = 1}^m \fix{{T_i}}\ne \emptyset$.

The iterative method for solving Problem \ref{Problem-VIP} is presented as follows.

\begin{algorithm}[H]
	\SetAlgoLined
	\vskip2mm
	\textbf{Initialization}: Given $\mu  \in (0,2\eta /{\kappa ^2})$, $\{ \beta _{n}\}_{n = 1}^\infty\subset(0,1]$, $\{ \varphi _{n}\}_{n = 1}^\infty\subset[0,\infty)$  and a positive sequence $\{\lambda_n\}_{n=1}^\infty$. Choose $x^1\in \Hi$ arbitrarily and set ${d^1} =  - F({x^1})$. \\
	
	\textbf{Iterative Steps}: For a current iterate $x^n\in \Hi$ ($n\in\mathbb{N}$), calculate as follows:
	
	\textbf{Step 1}. Compute $y^n$ and the step size as 
	$$y^n:=x^n + \mu\beta _{n}d^n$$
	and
	\begin{equation*}
	\sigma (y ^n):=\left\{ 
	\begin{array}{ll}
	\displaystyle\frac{\sum_{i=1}^{m}\langle Ty^n-S_{i-1}y^n,S_{i}y^n-S_{i-1}y^n\rangle 
	}{\Vert Ty^n-y^n\Vert ^{2}}, & \text{for \ }y^n\notin \bigcap\limits_{i = 1}^m \fix{{T_i}}, \\ 
	1, & \text{otherwise.}%
	\end{array}%
	\right. 
	\end{equation*}%
	
	
	\textbf{Step 2}. Compute the next iterate and the search direction as
	\begin{eqnarray}\label{ECCM}
	{x^{n + 1}} &:=& {T_m}({y ^n} + {\lambda _n}\sigma ({y ^n})(Ty ^n - y^n)), \\
	{d^{n + 1}}&:=&  - F({x^{n + 1}}) + {\varphi _{n + 1}}{d^n}.\nonumber
	\end{eqnarray} 
	
	Update $n:=n+1$ and return to \textbf{Step 1}.
	\caption{ESCoM-CGD}
	\label{algorithm}
	\vskip2mm
\end{algorithm}


\begin{remark}\label{rem8}
	\begin{itemize}
		\item[(i)] In the case of $m=1$, $\lambda_n\equiv1$, and $\sigma(y^n)\equiv1$, Algorithm \ref{algorithm} becomes HCGM considered in \cite{IY09}. Furthermore, if $\varphi_n\equiv0$, Algorithm \ref{algorithm} is the same as HSDM investigated by Yamada \cite{Y01}.
		\item[(ii)] If $F\equiv0$, Algorithm \ref{algorithm} forms a generalization of MECSPM \cite{CN19} in the sense of the operators $T_i, i=1,\ldots,m$, are assumed to be subgradient projections. Moreover, if the operator $T_m$ in (\ref{ECCM}) is omitted from the method, Algorithm \ref{algorithm} coincides with ECCM \cite{CC12}.
		\item[(iii)]  Note that Algorithm \ref{algorithm}  is not feasible in the sense that the generated sequence $\{ {x^n}\} _{n = 1}^\infty$ need not belong to the constrained set. Moreover,  the step size $\sigma(y^n)$ may have large values for some $n\in\N$. These situations may yield the instabilities of the method. To avoid this situation, let us observe that if the operator $T_m$ is the metric projection onto a nonempty closed convex and bounded set $C_m$, and the initial point $x^1$ is chosen from $C_m$, then the iterate ${x^n}\in C_m$  $(n\in\N)$, which subsequently yields the boundedness of $\{ {x^n}\} _{n = 1}^\infty$. In this case, even if we can not gain the feasibility of the method, it is very worth to note that the presence of $T_m$ in (\ref{ECCM}) ensure us that the generated sequence $\{ {x^n}\} _{n = 1}^\infty\subset C_m$, which may yield the numerical stabilities of the method, see \cite[Section 4]{CN19} further discussion and some numerical illustrations.
	\end{itemize}
\end{remark}

It is worth noting that the existence and uniqueness of the solution to Problem \ref{Problem-VIP} is guaranteed by the above conditions according to \cite[Theorem 2.3.3]{FP03}.
In order to analyze the main convergence theorem, we present a series of preliminary convergence results which is indicating some important properties of the sequences generated by Algorithm \ref{algorithm}. To begin with, the boundedness of the sequences is investigated in the following lemma.

\begin{lemma}\label{lemma1} Let the sequences $\{ {x^n}\} _{n = 1}^\infty$, $\{ {y^n}\} _{n = 1}^\infty$ and  $\{ {d^n}\} _{n = 1}^\infty$ be given by Algorithm \ref{algorithm}. Suppose that $\mathop {\lim }\limits_{n \to \infty }{\beta _n} = 0$, $\mathop {\lim }\limits_{n \to \infty } {\varphi _n} = 0$, and $\{ \lambda _{n}\}_{n = 1}^\infty \subset [\varepsilon ,2 - \varepsilon]$ for some constant $\varepsilon \in (0,1)$. If $\{ F({x^n})\} _{n = 1}^\infty$ is bounded, then the sequences $\{ {x^n}\} _{n = 1}^\infty$, $\{ {y^n}\} _{n = 1}^\infty$ and  $\{ {d^n}\} _{n = 1}^\infty$ are bounded.
\end{lemma}
\begin{proof}
	Assume that $\{ F({x^n})\} _{n = 1}^\infty$ is bounded. We first show that $\{d^n\}_{n=1}^\infty$ is bounded. Accordingly, the assumption $\mathop{\lim}\limits_{n\to\infty}\varphi_n=0$ yields that there exists $n_0\in\mathbb{N}$ such that $\varphi_n\le\frac{1}{2}$ for all $n\ge n_0$. Due to the boundedness of $\{F(x^n)\}$, we set $M_1:=\mathop{\sup}\limits_{n\ge1}\|F(x^n)\|<\infty$ and $M_2:=\max\{M_1,\|d^{n_0}\|\}$. It is obvious to see that $\|d^{n_0}\|\le 2M_2$. By the definition of $\{d^n\}_{n=1}^\infty$, for all $n\ge n_0$, we have
	\begin{eqnarray}
	\|d^{n+1}\|
	\le  \|-F(x^{n+1})\|+\varphi_{n+1}\|d^n\|
	\le  \|F(x^{n+1})\|+\frac{1}{2}\|d^n\|
	\le M_2+\frac{1}{2}\|d^n\|.\label{4.1}
	\end{eqnarray}
	Now, we claim that $\|d^n\|\le2M_2$ for all $n\ge n _0$. 
	For $n=n_0$, we immediately get $\|d^n\|=\|d^{n_0}\|\le2M_2$. 
	Let $n\ge n_0$ and $\|d^n\|\le2M_2$. We shall prove that $\|d^{n+1}\|\le2M_2.$
	By (\ref{4.1}), we have 
	$$\|d^{n+1}\|\le M_2+\frac{1}{2}\|d^n\|\le 2M_2.$$
	Thus $\|d^n\|\le2M_2$ for all $n\ge n_0$.
	Putting $M^*:=\max\{\|d^1\|,\|d^2\|,\ldots,\|d^{n_0-1}\|,2M_2\}$, we obtain that
	$\|d^n\|\le M^*,$ for all $n\in\mathbb{N}.$         
	Therefore $\{d^n\}_{n=1}^\infty$ is bounded.

	Next, we will show that $\{ {x^n}\} _{n = 1}^\infty$ is bounded. Let $\bar u\in \bigcap\limits_{i = 1}^m \fix{{T_i}}$ be given. According to Lemma \ref{lemma-CC12}(ii), it is worth noting here that $T_\sigma$ is a cutter. By utilizing the quasi-nonexpansivity of $T_m$ and the properties of $T_\sigma$ in Fact \ref{fact2}, for all $n\in\mathbb{N}$, we have
	\begin{eqnarray}
	\|{x^{n + 1}} - \bar u{\|^2} &=& \|{T_m}({y^n} + {\lambda _n}\sigma ({y^n})(T{y^n} - {y^n})) - \bar u{\|^2}\nonumber\\
	&\le& \|{y^n} + {\lambda _n}\sigma ({y^n})(T{y^n} - {y^n}) - \bar u{\|^2}\nonumber\\
	&=& \|{y^n} - \bar u{\|^2} + \lambda _n^2\|\sigma ({y^n})(T{y^n} - {y^n}){\|^2} \nonumber\\
	&&+ 2{\lambda _n}\langle {y^n} - \bar u,\sigma ({y^n})(T{y^n} - {y^n})\rangle\nonumber \\
	&=& \|{y^n} - \bar u{\|^2} + \lambda _n^2\|{T_\sigma }{y^n} - {y^n}{\|^2} + 2{\lambda _n}\langle {y^n} - \bar u,{T_\sigma }{y^n} - {y^n}\rangle\nonumber \\
	&\le& \|{y^n} - \bar u{\|^2} + \lambda _n^2\|{T_\sigma }{y^n} - {y^n}{\|^2} - 2{\lambda _n}\|{T_\sigma }{y^n} - {y^n}{\|^2}\nonumber\\
	&=& \|{y^n} - \bar u{\|^2} - {\lambda _n}(2 - {\lambda _n})\|{T_\sigma }{y^n} - {y^n}{\|^2}.\label{lem1-3}
	\end{eqnarray}
	Since $\{ \lambda _{n}\}_{n = 1}^\infty \subset [\varepsilon ,2 - \varepsilon]$ for some constant $\varepsilon \in (0,1)$, we obtain that
	\begin{equation}\label{lem1-4}
	\|{x^{n + 1}} - \bar u\| \le \|{y^n} - \bar u\|.
	\end{equation}
	Then, for all $n\ge2$, we have
	\begin{eqnarray}
	\|{y^n} - \bar u\| &=& \|{x^n} + \mu {\beta _{n}}{d^n} - \bar u\|\nonumber\\
	&=& \|{x^n} + \mu {\beta _{n}}( - F({x^n}) + {\varphi _n}{d^{n - 1}}) - \bar u\|\nonumber\\
	&=& \|({x^n} - \mu {\beta _{n}}F({x^n})) - (\bar u - \mu {\beta _{n }}F(\bar u)) + \mu {\beta _{n }}({\varphi _n}{d^{n - 1}} - F(\bar u))\|\label{lem1-5}\nonumber\\
	&\le& \|({x^n} - \mu {\beta _{n}}F({x^n})) - (\bar u - \mu {\beta _{n }}F(\bar u))\| + \mu {\beta _{n }}\|{\varphi _n}{d^{n - 1}} - F(\bar u)\|.\label{lem1-4*}
	\end{eqnarray}
	By using the inequalities (\ref{lem1-4}), (\ref{lem1-4*}) and Lemma \ref{yamada}, for all $n\ge2$, we obtain
	\begin{equation}
	\|{x^{n+1}} - \bar u\| \le (1 - {\beta _{n}}\tau )\|{x^n} - \bar u\| + \mu {\beta _{n}}\|{\varphi _n}{d^{n - 1}} - F(\bar u)\|,\nonumber
	\end{equation}
	where $\tau = 1 - \sqrt {1 + {\mu ^2}{\kappa ^2} - 2\mu \eta }  \in (0,1].$
	Accoding to the boundedness of $\{d^{n}\}_{n = 1}^\infty$, we set ${M_3}: = \mathop {\sup }\limits_{n \ge 1} \|{\varphi _n}{d^{n - 1}} - F(\bar u)\| < \infty$ and $M: = \max \{ {M_3},\|F(\bar u)\|\}$. The inequality above becomes
	\begin{equation}\label{1v}
	\|{x^{n + 1}} - \bar u\| \le (1 - {\beta _{n}}\tau )\|{x^n} - \bar u\| + {\beta _{n}}\tau \left( {\frac{{\mu M}}{\tau }} \right) \text{      for all $n\ge 2$.     }
	\end{equation}
	However, one can easily check that the inequality (\ref{1v}) also holds true for $n=1$. In the light of induction, we ensure that
	\[\|{x^n} - \bar u\| \le \max \{ \|{x^1} - \bar u\|,\frac{{\mu M}}{\tau }\} \text{     for all $n\in\mathbb{N}$.     }\]
	Thus $\{x^{n}\}_{n = 1}^\infty$ is bounded as desired. Consequently,  $\{y^{n}\}_{n = 1}^\infty$ is also bounded. 
	\qed
\end{proof}

Before continuing the analysis, for $n\in\mathbb{N}$ and $\bar u\in \bigcap\limits_{i = 1}^m \fix{{T_i}}$, let us denote the following terms:
\[\xi_n:={\mu ^2}\beta _{n}^2\|{d^n}{\|^2} + 2\mu {\beta _{n}}\|{x^n} - \bar u\|\|{d^n}\|\indent \text{      and         }\indent {\alpha _n}: = {\beta _{n}}\tau.\]
In particular, for $n\ge2$, we denote
\[{\delta _n}: = \frac{{2\mu }}{\tau }\left( {{\varphi _n}\langle {y^n} - \bar u,{d^{n - 1}}\rangle  + \langle {y^n} - \bar u, - F(\bar u)\rangle } \right).\]

The aforementioned notations give rise to the following lemmas which demonstate some crucial inequalities needed in proving our main convergence result.



\begin{lemma}\label{lemma2}Let the sequences $\{ {x^n}\} _{n = 1}^\infty$, $\{ {y^n}\} _{n = 1}^\infty$ and  $\{ {d^n}\} _{n = 1}^\infty$ be given by Algorithm \ref{algorithm}. Suppose that $\{ \lambda _{n}\}_{n = 1}^\infty \subset [\varepsilon ,2 - \varepsilon]$ for some constant $\varepsilon \in (0,1)$. Then, for all $n\in\mathbb{N}$ and $\bar u\in \bigcap\limits_{i = 1}^m \fix{{T_i}}$, there holds:
	\begin{eqnarray*}
		\|{x^{n + 1}} - \bar u{\|^2} \le \|{x^n} - \bar u{\|^2} - \frac{{{\lambda _n}(2 - {\lambda _n})}}{{4m}}\sum\limits_{i = 1}^m {\parallel {S_i}{y^n} - {S_{i - 1}}{y^n}{\parallel ^2}}  + {\xi _n}.
	\end{eqnarray*}
\end{lemma}


\begin{proof}
	By invoking the inequality (\ref{lem1-3}) and the definition of $T_\sigma$, we have 
	\begin{eqnarray}
	\|{x^{n + 1}} - \bar u{\|^2} &\le& \|{y^n} - \bar u{\|^2} - {\lambda _n}(2 - {\lambda _n})\|{T_\sigma }{y^n} - {y^n}{\|^2}\nonumber\\
	&\le& \|{x^n} + \mu {\beta _{n}}{d^n} - \bar u{\|^2} - {\lambda _n}(2 - {\lambda _n})\|{T_\sigma }{y^n} - {y^n}{\|^2}\nonumber\\
	&=& \|{x^n} - \bar u{\|^2} + {\mu ^2}\beta _{n }^2\|{d^n}{\|^2} + 2\mu {\beta _{n }}\langle {x^n} - \bar u,{d^n}\rangle\nonumber\\
	&&  - {\lambda _n}(2 - {\lambda _n})\|{T_\sigma }{y^n} - {y^n}{\|^2}\nonumber\\
	&\le& \|{x^n} - \bar u{\|^2} + {\mu ^2}\beta _{n }^2\|{d^n}{\|^2} + 2\mu\beta_n\|{x^n} - \bar u\|\|{d^n}\|\nonumber\\
	&& - {\lambda _n}(2 - {\lambda _n})\|{T_\sigma }{y^n} - {y^n}{\|^2}\nonumber\\
	&=& \|{x^n} - \bar u{\|^2} - {\lambda _n}(2 - {\lambda _n})\|{T_\sigma }{y^n} - {y^n}{\|^2} + {\xi _n}\nonumber\\
	&=& \|{x^n} - \bar u{\|^2} - {\lambda _n}(2 - {\lambda _n})\sigma^2 ({y^n})\|T{y^n} - {y^n}{\|^2} + {\xi _n}.\nonumber
	\end{eqnarray}
	Thanks to Lemma \ref{lemma-CC12}(i), we finally have
	{\small	\begin{eqnarray}
		\|{x^{n + 1}} - \bar u{\|^2} &\le& \|{x^n} - \bar u{\|^2} - {\lambda _n}(2 - {\lambda _n})\frac{{\frac{1}{4}{{\left( {\sum\limits_{i = 1}^m {\|{S_i}{y^n} - {S_{i - 1}}{y^n}{\|^2}} } \right)}^2}}}{{\|T{y^n} - {y^n}{\|^4}}}\|T{y^n} - {y^n}{\|^2} + {\xi _n}\nonumber\\
		&=& \|{x^n} - \bar u{\|^2} - {\lambda _n}(2 - {\lambda _n})\frac{{\frac{1}{4}{{\left( {\sum\limits_{i = 1}^m {\|{S_i}{y^n} - {S_{i - 1}}{y^n}{\|^2}} } \right)}^2}}}{{\|T{y^n} - {y^n}{\|^2}}} + {\xi _n}\nonumber\\
		&=& \|{x^n} - \bar u{\|^2} - \frac{{{\lambda _n}(2 - {\lambda _n})}}{{4m}}\sum\limits_{i = 1}^m {\|{S_i}{y^n} - {S_{i - 1}}{y^n}{\|^2}}  + {\xi _n},\nonumber
		\end{eqnarray}
	}
	which completes the proof.
	\qed
\end{proof}

\begin{lemma}\label{lemma4} Let the sequences $\{ {x^n}\} _{n = 1}^\infty$, $\{ {y^n}\} _{n = 1}^\infty$ and  $\{ {d^n}\} _{n = 1}^\infty$ be given by Algorithm \ref{algorithm}. Suppose that $\{ \lambda _{n}\}_{n = 1}^\infty \subset [\varepsilon ,2 - \varepsilon]$ for some constant $\varepsilon \in (0,1)$. Then, for all $n\ge2$ and $\bar u\in \bigcap\limits_{i = 1}^m \fix{{T_i}}$, there holds:
	\begin{eqnarray*}
		\|{x^{n + 1}} - \bar u{\|^2} \le (1 - {\alpha _n})\|{x^n} - \bar u{\|^2} + {\alpha _n}{\delta _n}.
	\end{eqnarray*}
	
\end{lemma}	
\begin{proof}
	By utilizing the inequalities (\ref{lem1-4}), (\ref{lem1-5}), the fact that $\|x+y\|^2\le\|x\|^2+2\langle y, x+y\rangle$, for all $x,y\in\Hi$, and Lemma \ref{yamada}, for all $n\ge2$, we have
	\begin{eqnarray}
	\|{x^{n + 1}} - \bar u{\|^2} &\le& \|{y^n} - \bar u{\|^2}\nonumber\\
	&\le& \|({x^n} - \mu {\beta _{n}}F({x^n})) - (\bar u - \mu {\beta _{n }}F(\bar u)) + \mu {\beta _{n}}({\varphi _n}{d^{n - 1}} - F(\bar u)){\|^2}\nonumber\\
	&\le& \|({x^n} - \mu {\beta _{n }}F({x^n})) - (\bar u - \mu {\beta _{n }}F(\bar u)){\|^2}  \nonumber \\
	&&+ 2\langle {x^n} - \mu {\beta _{n }}F({x^n}) - \bar u+ \mu {\beta _{n }}{\varphi _n}{d^{n - 1}},\mu {\beta _{n }}({\varphi _n}{d^{n - 1}} - F(\bar u))\rangle\nonumber \\
	&\le& (1 - {\beta _{n }}\tau )\|{x^n} - \bar u{\|^2} \nonumber\\
	&&+ 2\mu {\beta _{n }}\langle {x^n} + \mu {\beta _{n }}( - F({x^n}) + {\varphi _n}{d^{n - 1}}) - \bar u,{\varphi _n}{d^{n - 1}} - F(\bar u)\rangle\nonumber \\
	&=& (1 - {\beta _{n}}\tau )\|{x^n} - \bar u{\|^2} + 2\mu {\beta _{n}}\langle {y^n} - \bar u,{\varphi _n}{d^{n - 1}} - F(\bar u)\rangle \nonumber\\
	&=& (1 - {\beta _{n}}\tau )\|{x^n} - \bar u{\|^2} + 2\mu {\beta _{n}}{\varphi _n}\langle {y^n} - \bar u,{d^{n - 1}}\rangle  + 2\mu {\beta _{n}}\langle {y^n} - \bar u, - F(\bar u)\rangle\nonumber \\
	&=& (1 - {\beta _{n}}\tau )\|{x^n} - \bar u{\|^2} + {\beta _{n}}\tau \left[ {\frac{{2\mu }}{\tau }\left( {{\varphi _n}\langle {y^n} - \bar u,{d^{n - 1}}\rangle  + \langle {y^n} - \bar u, - F(\bar u)\rangle } \right)} \right]\nonumber\\
	&=& (1 - {\alpha _n})\|{x^n} - \bar u{\|^2} + {\alpha _n}{\delta _n},\nonumber
	\end{eqnarray}
	which completes the proof.
	\qed
\end{proof}


We present the following lemma which is an important tool for proving our main result. A proof of the lemma can be found in \cite[Lemma 2.5]{X02}. 
\begin{lemma}\label{xu}
	Let $\{ {a_n}\} _{n = 1}^\infty$ be a sequence of nonnegative real numbers such that
	${a_{n + 1}} \le (1 - {\alpha _n}){a_n} + {\alpha _n}{\delta _n},$
	where the sequences $\{ {\alpha _n}\} _{n = 1}^\infty\subset [0,1]$ and $\{ {\delta _n}\} _{n = 1}^\infty\subset\mathbb{R}$ satisfy
	$\sum\limits_{n = 1}^\infty  {{\alpha _n}}  = \infty$ and $\limsup\limits_{n\rightarrow0}\delta_n\le 0$. Then $\mathop {\lim }\limits_{n \to \infty } {a_n} = 0$.
\end{lemma}

The following theorem is our main convergence result.

\begin{theorem}\label{main-thm} Let the sequence $\{ {x^n}\} _{n = 1}^\infty$ be given by Algorithm \ref{algorithm}. Suppose that $\mathop {\lim }\limits_{n \to \infty }{\beta _n} = 0$, $\sum\limits_{n = 1}^\infty  {{\beta _n}}= \infty$, $\mathop {\lim }\limits_{n \to \infty } {\varphi _n} = 0$, and $\{ \lambda _{n}\}_{n = 1}^\infty \subset [\varepsilon ,2 - \varepsilon]$ for some constant $\varepsilon \in (0,1)$. If $\{ F({x^n})\} _{n = 1}^\infty$ is bounded and  $\{ {T_i}\} _{i = 1}^m$ satisfies the DC principle, then the sequence $\{ {x^n}\} _{n = 1}^\infty$ converges strongly to $\bar u$, the unique solution of Problem \ref{Problem-VIP}.
\end{theorem}

\begin{proof}
	Assume that $\{ F({x^n})\} _{n = 1}^\infty$ is bounded and  $\{ {T_i}\} _{i = 1}^m$ satisfies the DC principle. For simplicity, we denote ${a_n}: = \|{x^n} - \bar u{\|^2}$. Due to Lemma \ref{lemma1} and the assumption $\mathop {\lim }\limits_{n \to \infty }\beta_n=0$, we obtain
	$$\mathop {\lim }\limits_{n \to \infty }\xi_n=0.$$
	To prove the strong convergence of the theorem, we consider the following two cases of the sequence  $\{a_n\} _{n = 1}^\infty$ according to its behavior.
	
	\indent\textbf{Case 1.} Suppose that there exists ${n_0}\in\mathbb{N}$ such that ${a_{n + 1}} < {a_n}$ for all $n \ge {n_0}$. It is clear that $\{a_n\} _{n = 1}^\infty$ is convergent. By utilizing Lemma \ref{lemma2} and the assumption $\mathop {\lim }\limits_{n \to \infty }\xi_n=0$, we obtain
	\begin{eqnarray*}
		0 &\le& \limsup_{n \to \infty } \frac{{{\lambda _n}(2 - {\lambda _n})}}{{4m}}\sum\limits_{i = 1}^m {\|{S_i}{y^n} - {S_{i - 1}}{y^n}{\|^2}}\\
		&\le&\limsup_{n \to \infty } \left( {{a_n} - {a_{n + 1}}} + \xi_n\right) = \lim_{n \to \infty } a_n - \lim_{n \to \infty }a_{n+1}+\lim_{n \to \infty }\xi_n= 0.\nonumber
	\end{eqnarray*}
	Thus, we have
	\begin{equation*}\label{18}
	\mathop {\lim }\limits_{n \to \infty } \frac{{{\lambda _n}(2 - {\lambda _n})}}{{4m}}\sum\limits_{i = 1}^m {\|{S_i}{y^n} - {S_{i - 1}}{y^n}{\|^2}}  = 0.
	\end{equation*}
	Recalling that ${\lambda _n} \in [\varepsilon ,2 - \varepsilon ]$ for an arbitrary constant $\varepsilon  \in (0,1),$ we then have
	\[{\lambda _n} \le 2 - \varepsilon  \Rightarrow \varepsilon  \le 2 - {\lambda _n} \Rightarrow {\lambda _n}\varepsilon  \le {\lambda _n}(2 - {\lambda _n}) \Rightarrow {\varepsilon ^2} \le {\lambda _n}(2 - {\lambda _n}),\]
	and hence
	\[\mathop {\lim }\limits_{n \to \infty } \sum\limits_{i = 1}^m {\|{S_i}{y^n} - {S_{i - 1}}{y^n}{\|^2}}  = 0,\]
	which implies that,  for all $i=1,2,...,m,$
	\begin{equation}\label{20}
	\mathop {\lim }\limits_{n \to \infty } \|{S_i}{y^n} - {S_{i - 1}}{y^n}\| = 0.     
	\end{equation}
	On the other hand, since $\{ {y^n}\} _{n = 1}^\infty$ is a bounded sequence, so is the sequence $\{ \langle {y^n} - \bar u, - F(\bar u)\rangle \} _{n = 1}^\infty$. Now, let $\{ {y^{n_k}}\} _{k = 1}^\infty$ be a subsequence of $\{ {y^n}\} _{n = 1}^\infty$ such that
	\[\limsup\limits_{n\rightarrow\infty}\langle {y^n} - \bar u, - F(\bar u)\rangle  = \mathop {\lim }\limits_{k \to \infty } \langle {y^{n_k}} - \bar u, - F(\bar u)\rangle. \]
	Due to the boundedness of the sequence $\{ {y^{n_k}}\} _{k = 1}^\infty$, there exists a weakly cluster point $z\in\Hi$ and a subsequence $\{ {y^{{n_{{k_j}}}}}\} _{j = 1}^\infty$ of $\{ {y^{n_k}}\} _{k = 1}^\infty$ such that ${y^{{n_{{k_j}}}}}\rightharpoonup z\in\Hi$. According to (\ref{20}), let us note that
	\begin{equation*}\label{wc-ii-2}
	\mathop {\lim }\limits_{j \to \infty }\| ({T_1} - Id){y^{{n_{{k_j}}}}}\|=  \mathop {\lim }\limits_{j \to \infty }\| {S_1}{y^{{n_{{k_j}}}}} - {S_0}{y^{{n_{{k_j}}}}}\| = 0.
	\end{equation*}
	Then the DC principle of $T_1$ yields that
	$z\in\fix T_1.$
	Further, we note that the assumption ${y^{{n_{{k_j}}}}}\rightharpoonup z$  and the fact that
	$\mathop {\lim }\limits_{j \to \infty }\| {T_1}{y^{{n_{{k_j}}}}} - {y^{{n_{{k_j}}}}} \| =0$
	lead to ${T_1}{y^{{n_{{k_j}}}}} \rightharpoonup z.$
	Furthermore, we observe that
	\begin{equation*}\label{wc-ii-6}
	\mathop {\lim }\limits_{j \to \infty }\| ({T_2} - Id){T_1}{y^{{n_{{k_j}}}}}\|  = \mathop {\lim }\limits_{j \to \infty } \| {S_2}{y^{{n_{{k_j}}}}} - {S_1}{y^{{n_{{k_j}}}}}\| = 0.
	\end{equation*}
	By invoking the DC principle of $T_2$, we then obtain 
	$z\in\fix T_2.$
	By continuing the same argument used in the above proving lines, we obtain that $z\in\fix T_i$ for all $i=1,2,\ldots,m,$
	that is $z \in\bigcap_{i=1}^m\fix T_i$. 
	As $\bar u$ is the unique solution to Problem \ref{Problem-VIP}, we have
	\begin{eqnarray}\label{limsupvar}\limsup_{n\rightarrow\infty} \langle {y^n} - \bar u, - F(\bar u)\rangle  &=& \lim_{k \to \infty } \langle {y^{n_k}} - \bar u, - F(\bar u)\rangle\nonumber\\
	&=& \lim_{j \to \infty } \langle {y^{{n_{{k_j}}}}} - \bar u, - F(\bar u)\rangle  = \langle z - \bar u, - F(\bar u)\rangle  \le 0.
	\end{eqnarray}
	In view of $\delta_n$, we note that
	\begin{eqnarray}
	{\delta _n} \le \frac{{2\mu K}}{\tau }{\varphi _n} + \frac{{2\mu }}{\tau }\langle {y^n} - \bar u, - F(\bar u)\rangle, \nonumber
	\end{eqnarray}
	where $K:=\mathop {\sup }\limits_{n \ge 2} \langle {y^n} - \bar u,{d^{n - 1}}\rangle   <  + \infty$. Therefore, $\mathop {\lim }\limits_{n \to \infty } {\varphi _n} = 0$ and the inequality (\ref{limsupvar}) lead to 
	\begin{eqnarray}\label{limdelta}
	\limsup\limits_{n\rightarrow\infty}{\delta _n} &\le& \frac{{2\mu K}}{\tau }\mathop {\lim }\limits_{n \to \infty } {\varphi _n} + \frac{{2\mu }}{\tau }\limsup\limits_{n\rightarrow\infty}\langle {y^n} - \bar u, - F(\bar u)\rangle\le 0.\label{21-1}
	\end{eqnarray}
	According to Lemma \ref{lemma4}, we have, for all $n\ge 2$, that
	$$a_{n+1} \le (1 - {\alpha _n})a_n + {\alpha _n}{\delta _n}.$$
	To reach the conclusion of this case, we observe that $\{\alpha_n\}_{n=1}^\infty\subset[0,1]$ and $\sum\limits_{n = 1}^\infty  {{\alpha _n}}=\infty$ which are following the assumptions of $\beta_n$ and the property of $\tau$. Therefore, by applying this and the relation (\ref{limdelta}), Lemma \ref{xu} yields that $\mathop {\lim }\limits_{n \to \infty }\|{x^n}-\bar u\| = 0$.
	
	\textbf{Case 2.} Suppose that for any $n_0$, there exists an integer $k\ge n_0$ such that ${a_{{k}}} \le {a_{{k} + 1}}$. For $n$ large enough, we define a set of indexes by 
	\begin{equation*}\label{jn}
	{J_n}: = \left\{ {k \in [{n_0},n]:{a_k} \le {a_{k + 1}}} \right\}.
	\end{equation*}
	Also, for each $n\ge n_0$, we denote
	\begin{equation*}\label{nun}
	\nu(n):=\max J_n.
	\end{equation*}
	From the above definitions, we observe that $J_n$ is nonempty as there is an $n_0\in J_n$. Due to ${J_n} \subset {J_{n + 1}}$, we get that $\{\nu(n)\}_{n\ge n_0}$ is nondecreasing and $\nu(n)\to\infty$ as $n\to\infty$. Furthermore,  it is clear that,  for all $n\ge n_0$,
	\begin{equation}\label{22}
	a_{\nu(n)}\le a_{\nu(n)+1}.
	\end{equation}
	Now, let us notice from the definition of $J_n$ that, for all $n\ge n_0$, we have $\nu(n)\le n$ which can be considered in the following cases:
	If $\nu(n)= n$, we have $a_{n}=a_{\nu(n)}\le a_{\nu(n)+1}$. 
	If $\nu(n)= n-1$, we have $a_n=a_{\nu(n) +1}$. 
	If $\nu(n)<n-1$, we have ${a_{\nu (n) + 1}} > {a_{\nu (n) + 2}} > \cdots > {a_{n - 1}} > {a_n}$ which is followed by the fact that whenever  we set ${a_{\nu (n) + 1}} \le {a_{\nu (n) + 2}}$, the definition of $S_n$ yields that $\nu (n) + 1\in S_n$. However, we know that $\nu(n)=\max{J_n}$. Thus the assumption ${a_{\nu (n) + 1}} \le {a_{\nu (n) + 2}}$ leads to a contradiction. This similar argument happens to the other terms as well. Therefore, by the aforementioned cases, we obtain, for all $n\ge n_0$, that
	\begin{equation}\label{23}
	a_{n}\le a_{\nu(n)+1}.
	\end{equation}
	Next, utilizing Lemma \ref{lemma2} and the inequality (\ref{22}) lead to 
	\[0 \le {a_{\nu (n) + 1}} - {a_{\nu(n)}} \le  - \frac{{{\lambda _{\nu(n)}}\left(2 - {\lambda _{\nu(n)}}\right)}}{{4m}}\sum\limits_{i = 1}^m {\| {S_i}{y^{\nu(n)}} - {S_{i - 1}}{{y^{\nu(n)}}}{\| ^2}}  + {\xi _{\nu(n)}},\]
	and hence
	\[\frac{{{\lambda _{\nu(n)}}\left(2 - {\lambda _{\nu(n)}}\right)}}{{4m}}\sum\limits_{i = 1}^m {\| {S_i}{y^{\nu(n)}} - {S_{i - 1}}{{y^{\nu(n)}}}{\| ^2}}\leq{\xi _{\nu(n)}},\]
	for all $n\ge n_0$.
	According to the assumption $\mathop{\lim }\limits_{n \to \infty }\xi_{\nu(n)}=0$ and the fact that $\varepsilon ^2\le{\lambda _{\nu(n)}}(2 - {\lambda _{\nu(n)}})$, we obtain
	\begin{equation}\label{24}
	\mathop {\lim }\limits_{n \to \infty }\|{S_i}{y^{\nu (n)}} - {S_{i - 1}}{y^{\nu (n)}}\| = 0.
	\end{equation}
	Now, let $\{ {y^{{\nu(n_k)}}}\} _{k = 1}^\infty  \subset \{ {y^{\nu(n)}}\} _{n = 1}^\infty$ be a subsequence such that
	\[\limsup\limits_{n\rightarrow\infty}\langle {y^{\nu(n)}} - \bar u, - F(\bar u)\rangle  = \mathop {\lim }\limits_{k \to \infty } \langle {y^{{\nu(n_k)}}} - \bar u, - F(\bar u)\rangle.\]
	By proceeding the similar argument to those used in {\bf Case 1}, the relation (\ref{24}) and the DC principle of each $T_i$ yields that, for any subsequence $\{ {y^{{\nu(n_{{k_j}})}}}\} _{j = 1}^\infty$ of  $\{ {y^{{\nu{(n_k)}}}}\} _{k = 1}^\infty$, we get that
	$ {y^{{\nu(n_{{k_j}})}}}\rightharpoonup z\in\bigcap_{i=1}^m\fix T_i$. Furthermore, we have
	\begin{eqnarray*}\limsup\limits_{n\rightarrow\infty}\langle {y^{\nu (n)}} - \bar u, - F(\bar u)\rangle  &=& \mathop {\lim }\limits_{k \to \infty } \langle {y^{\nu ({n_k})}} - \bar u, - F(\bar u)\rangle\\
		&=& \mathop {\lim }\limits_{j \to \infty } \langle {y^{\nu ({n_{{k_j}}})}} - \bar u, - F(\bar u)\rangle  = \langle z - \bar u, - F(\bar u)\rangle  \le 0.
	\end{eqnarray*}
	As a result, we simultaneously obtain
	\begin{equation}\label{25}
	\limsup\limits_{n\rightarrow\infty}\delta_{\nu(n)}\le 0.
	\end{equation}
	In the light of Lemma \ref{lemma4}, we have 
	\[ {a_{\nu (n) + 1}} \le \left( {1 - {\alpha _{\nu (n)}}} \right){a_{\nu (n)}} + {\alpha _{\nu (n)}}{\delta _{\nu (n)}},\]
	and hence
	\begin{eqnarray}
	{a_{\nu (n) + 1}} - {a_{\nu (n)}} &\le& {\alpha _{\nu (n)}}\left( {{\delta _{\nu (n)}} - {a_{\nu (n)}}} \right). \nonumber
	\end{eqnarray}
	Since $\alpha _{\nu (n)}>0$, we obtain 
	\[{a_{\nu (n)}} \le {\delta _{\nu (n)}}.\]
	Thanks to (\ref{25}), we have
	\[\limsup\limits_{n\rightarrow\infty}{a_{\nu (n)}} \le \limsup\limits_{n\rightarrow\infty}{\delta _{\nu (n)}} \le 0,\]
	which leads to 
	\[\mathop {\lim }\limits_{n \to \infty } {a_{\nu (n)}} = 0.\]
	By utilizing the inequality (\ref{23}) together with this, we have
	\[0 \le \limsup\limits_{n\rightarrow\infty}{a_n} \le \limsup\limits_{n\rightarrow\infty}{a_{\nu (n) + 1}} = 0.\]
	Hence, we finally obtain that
	$\mathop {\lim }\limits_{n \to \infty } {a_n}  = 0$ as desired.
	\qed
\end{proof}

\begin{remark} 
	\begin{itemize} 
		\item[(i)]  The step-size sequences  $\{\varphi_n\}_{n=1}^\infty$ and $\{\beta_n\}_{n=1}^\infty$  in Theorem \ref{main-thm}  are, for instance,  $\varphi_n=\frac{1}{(n+1)^a}$ with $a>0$ and $\beta_n=\frac{1}{(n+1)^b}$ with $0<b\leq1$ for all $n\in\N$.
		\item[(ii)] It can be noted that the DC principle  assumed in Theorem \ref{main-thm} will be satisfying in many cases, for instance,  the operators $T_i, i=1,\ldots,m$, are nonexpansive, or, in particular,  the metric projections onto closed convex sets. Moreover, this still holds true when the operators $T_i, i=1,\ldots,m$, are subgradient projections of  continuous convex functions which are Lipschitz continuous on bounded subsets which is further discussed in the next section. 
	\end{itemize}
\end{remark}

\section{Variational Inequality Problem with Functional Constraints}

In this section, we will consider the solving of the variational inequality problem over the finite family of continuous convex functional constraints and a simple closed convex and bounded constraint by applying the results obtained in the previous section.

Let $C_i:=\{x\in\Hi:c_i(x)\leq0\}$ be a sublevel set of a continuous and convex function $c_i:\Hi\to\R$, $i=1,\ldots,m-1$, and $C_m\subset\Hi$ be a simple closed convex and bounded set. Let $F:\Hi\to\Hi$ be $\eta$-strongly monotone and $\kappa$-Lipschitz continuous, we consider the variational inequality of finding a point $\bar u\in \bigcap\limits_{i = 1}^m C_i$ such that 
\begin{eqnarray}\label{vip-2}\langle F({\bar u}),z- {\bar u}\rangle  \ge 0 \indent\text{    for all } z\in \bigcap\limits_{i = 1}^m C_i.  
\end{eqnarray}

Assume that $\bigcap\limits_{i = 1}^m C_i\neq\emptyset$. Let us consider, for each $i=1,\ldots,m-1$, since each $C_i$  is the sublevel set of the function $c_i$, we define the operator $T_i:\Hi\to\Hi$ to be a subgradient projection relative to $c_i$, $P_{c_i}:\Hi\to\Hi$, namely, for  every $x\in\Hi$, 
\begin{equation*}
P_{c_i}(x):=\left \{
\begin{aligned}
&x-\frac{c_i(x)}{\|g_i(x)\|^2}g_i(x) \indent \textrm{if } c_i(x)>0,\\
&x\hspace{3.05cm} \textrm{ otherwise,}
\end{aligned} \right.
\end{equation*}  
where $g_i(x)\in\partial c_i(x):=\{g\in\Hi:\<g,y-x\>\leq c_i(y)-c_i(x),\forall y\in\Hi\}$, is a subgradient of the function $c_i$ at the point $x$. Since $c_i, i=1,\ldots,m-1$, are continuous and convex, we ensure that the subdifferential sets $\partial c_i(x), i=1,\ldots,m-1$, are nonempty, for every $x\in\Hi$, see \cite[Proposition 16.17]{BC17}. Note that the subgradient projection $P_{c_i}$ is a cutter and $\fix P_{c_i}=C_i$, for all $i=1,\ldots,m-1$, see \cite[Lemma 4.2.5 and Corollary 2.4.6]{C12} 

Moreover, since $C_m$ is the nonempty closed convex and bounded, we define the operator $T_m:\Hi\to\Hi$ to be a metric projection onto $C_{m}$ written by $\mathrm{proj}_{C_m}:\Hi\to\Hi$, i.e., for every $x\in\Hi$, we have $$\|x-\mathrm{proj}_{C_m}x\|= \inf_{y\in C_m} \|x-y\|.$$
Note that the metric projection $\mathrm{proj}_{C_m}$ is also a cutter and $\fix \mathrm{proj}_{C_m}=C_m$, see \cite[Theorem 2.2.21]{C12}. 
These mean that the operators $T_i, i=1,\ldots,m$, are cutters and $\bigcap\limits_{i = 1}^m \fix T_i\neq\emptyset$.

Now, in order to construct an iterative method for solving the problem (\ref{vip-2}), we recall the notations $T := {T_m}{T_{m - 1}}\dots {T_1},$
$S_0:=Id$, and $S_i:=T_iT_{i-1}\dots T_1, i=1,2,\dots,m$. Furthermore,  for every $x\in\Hi$, we denote $u_i:=S_ix$, and $v_i:=u_i-u_{i-1}$. Thus, we have $u_0=x$ and $u_i=T_iu_{i-1}, i=1,2,\dots,m$. Firstly, let us note from \cite[Remark 10]{CC12} that
$$\sum_{i=1}^m\<v_i+v_{i+1}+\cdots+v_m,v_i\>=\sum_{i=1}^m\<v_1+v_{2}+\cdots+v_i,v_i\>.$$
It follows that, for every $x\in\Hi$, 
\begin{eqnarray*}
	\sum_{i=1}^m\<Tx-S_{i-1}x,S_ix-S_{i-1}x\>
	&=&\sum_{i=1}^m\<u_m-u_{i-1},u_i-u_{i-1}\>\\
	&=&\sum_{i=1}^m\<v_i+v_{i+1}+\cdots+v_m,v_i\>\\
	&=&\sum_{i=1}^m\<v_1+v_{2}+\cdots+v_i,v_i\>\\
	&=&\sum_{i=1}^m\<u_i-u_0,u_i-u_{i-1}\>\\
	&=&\sum_{i=1}^m\<T_iu_{i-1}-x,T_iu_{i-1}-u_{i-1}\>
\end{eqnarray*}

Now, for every $i=1,\ldots,m-1$, and $x\notin\bigcap\limits_{i = 1}^m \fix T_i$, we note that
\begin{equation*}
u_i=T_iu_{i-1}=u_{i-1}-\frac{\max\{c_i(u_{i-1}),0\}}{\|g_i(u_{i-1})\|^2}g_i(u_{i-1}),
\end{equation*}  
where $g_i(u_{i-1})$ is a subgradient of the function $c_i$ at the point $u_{i-1}$. For simplicity, we  use throughout  the convention that  $\frac{\max\{c_i(u_{i-1}),0\}}{\|g_i(u_{i-1})\|}=0$ whenever $\max\{c_i(u_{i-1}),0\}=0$. Subsequently, we have
{\small\begin{eqnarray*}
		\<T_iu_{i-1}-x,T_iu_{i-1}-u_{i-1}\>&=&\<T_iu_{i-1}-x,u_{i-1}-\frac{\max\{c_i(u_{i-1}),0\}}{\|g_i(u_{i-1})\|^2}g_i(u_{i-1})-u_{i-1}\>\\
		&=&-\frac{\max\{c_i(u_{i-1}),0\}}{\|g_i(u_{i-1})\|^2}\<T_iu_{i-1}-x,g_i(u_{i-1})\>\\
		&=&-\frac{\max\{c_i(u_{i-1}),0\}}{\|g_i(u_{i-1})\|^2}\<u_{i-1}-\frac{\max\{c_i(u_{i-1}),0\}}{\|g_i(u_{i-1})\|^2}g_i(u_{i-1})-x,g_i(u_{i-1})\>\\
		&=&-\frac{\max\{c_i(u_{i-1}),0\}}{\|g_i(u_{i-1})\|^2}\<u_{i-1}-x,g_i(u_{i-1})\>\\
		&&+\left(\frac{\max\{c_i(u_{i-1}),0\}}{\|g_i(u_{i-1})\|^2}\right)^2\<g_i(u_{i-1}),g_i(u_{i-1})\>\\
		&=&-\frac{\max\{c_i(u_{i-1}),0\}}{\|g_i(u_{i-1})\|^2}\<u_{i-1}-x,g_i(u_{i-1})\>+\left(\frac{\max\{c_i(u_{i-1}),0\}}{\|g_i(u_{i-1})\|}\right)^2
	\end{eqnarray*}
}
On the other hand, for every $x\notin\fix T_m$, we note that
$$\<\mathrm{proj}_{C_m}u_{m-1}-x,\mathrm{proj}_{C_m}u_{m-1}-u_{m-1}\>.$$
Therefore, the step-size function $\sigma:\Hi\to[0,+\infty)$ which is defined in (\ref{sigma}) can be written as
\begin{equation*}
\sigma (x):=\left\{ 
\begin{array}{ll}
\<\mathrm{proj}_{C_m}u_{m-1}-x,\mathrm{proj}_{C_m}u_{m-1}-u_{m-1}\>\\
-\sum_{i=1}^{m-1}\frac{\max\{c_i(u_{i-1}),0\}}{\|g_i(u_{i-1})\|^2}\<u_{i-1}-x,g_i(u_{i-1})\>\\
+\sum_{i=1}^{m-1}\left(\frac{\max\{c_i(u_{i-1}),0\}}{\|g_i(u_{i-1})\|}\right)^2, & \text{for \ }x\notin C, \\ 
1, & \text{otherwise.}%
\end{array}%
\right. 
\end{equation*}%
According to the above convention and Lemma \ref{lemma-CC12}(i), we can ensure that the step-size function $\sigma (x)$ is well-defined and nonnegative which is bounded from below by $\frac{1}{2m}$, for every $x\in\Hi$.

Now, we are in position to propose the method for solving the problem (\ref{vip-2}) as the following algorithm.\\

\begin{algorithm}[H]
	\SetAlgoLined
	\vskip2mm
	\textbf{Initialization}: Given $\mu  \in (0,2\eta /{\kappa ^2})$, $\{ \beta _{n}\}_{n = 1}^\infty\subset(0,1]$, $\{ \varphi _{n}\}_{n = 1}^\infty\subset[0,\infty)$  and a positive sequence $\{\lambda_n\}_{n=1}^\infty$. Choose $x^1\in C_m$ arbitrarily and set ${d^1} =  - F({x^1})$. \\
	
	\textbf{Iterative Steps}: For a given current iterate $x^n\in C_m$ ($n\in\mathbb{N}$), calculate as follows:
	
	\textbf{Step 1}. Compute $y^n$ as
	$$y^n:=x^n + \mu\beta _{n}d^n.$$
	
	\textbf{Step 2}. Set $u^n_0:=y^n$ and compute the estimates
	$$u^n_{i}=u^n_{i-1}-\frac{\max\{c_i(u^n_{i-1}),0\}}{\|g_i(u^n_{i-1})\|^2}g_i(u^n_{i-1}),\hspace{0.5cm}\indent i=1,\ldots,m-1,$$
	where $g_i(u^n_{i-1})$ is a subgradient of the function $c_i$ at the point $u^n_{i-1}$, and subsequently compute
	$$u^n_{m}:=\mathrm{proj}_{C_m}u^n_{m-1}.$$
	
	\textbf{Step 3}.  Compute a step size as
	\begin{equation*}
	\sigma (y^n):=\left\{ 
	\begin{array}{ll}
	\<\mathrm{proj}_{C_m}u^n_{m-1}-y^n,\mathrm{proj}_{C_m}u^n_{m-1}-u^n_{m-1}\>\\
	-\sum_{i=1}^{m-1}\frac{\max\{c_i(u^n_{i-1}),0\}}{\|g_i(u^n_{i-1})\|^2}\<u^n_{i-1}-x,g_i(u^n_{i-1})\>\\
	+\sum_{i=1}^{m-1}\left(\frac{\max\{c_i(u^n_{i-1}),0\}}{\|g_i(u^n_{i-1})\|}\right)^2, & \text{for \ }y^n\notin C, \\ 
	1, & \text{otherwise.}%
	\end{array}%
	\right. 
	\end{equation*}%
	
	\textbf{Step 4}. Compute a next iterate and a search direction as
	\begin{equation*}
	\left \{
	\begin{aligned}
	&{x^{n + 1}} := {\mathrm{proj}_{C_m}}({y ^n} + {\lambda _n}\sigma ({y ^n})(u^n_{m} - y^n)), \\
	&{d^{n + 1}}: =  - F({x^{n + 1}}) + {\varphi _{n + 1}}{d^n}.\\
	\end{aligned} \right.
	\end{equation*}

	Update $n:=n+1$ and return to \textbf{Step 1}.
	\caption{ESCoM-CGD for VIP with functional constraints}
	\label{algorithm-2}
	\vskip2mm
\end{algorithm}


\begin{remark} Observe that Algorithm \ref{algorithm-2} is nothing else than a particular case of ESCoM-CGD (Algorithm \ref{algorithm}). Moreover, as we have mentioned in Remark \ref{rem8} (iii), we underline here again that the initial point $x^1$ in Algorithm \ref{algorithm-2} is particularly chosen in the nonempty closed convex and bounded subset $C_m$ rather than in the whole space $\Hi$, and the generated iterates $x^n,n\geq2$, are projected into the subset $C_m$. These are done in order to ensure  the boundedness of the generated sequence $\{ {x^n}\} _{n = 1}^\infty$.
\end{remark}

The following corollary is a consequence of Theorem \ref{main-thm}.

\begin{corollary}\label{cor} Let the sequence $\{ {x^n}\} _{n = 1}^\infty$ be given by Algorithm \ref{algorithm-2}. Suppose that $\mathop {\lim }\limits_{n \to \infty }{\beta _n} = 0$, $\sum\limits_{n = 1}^\infty  {{\beta _n}}= \infty$, $\mathop {\lim }\limits_{n \to \infty } {\varphi _n} = 0$, and $\{ \lambda _{n}\}_{n = 1}^\infty \subset [\varepsilon ,2 - \varepsilon]$ for some constant $\varepsilon \in (0,1)$. If one of the following conditions hold:
	\begin{itemize}
		\item[(i)] The functions $c_i, i=1,\ldots,m-1$, are Lipschitz continuous relative to every bounded subset of $\Hi$;
		\item[(ii)] The functions $c_i, i=1,\ldots,m-1$, are bounded on every bounded subset of $\Hi$;
		\item[(iii)] The subdifferentials $\partial c_i, i=1,\ldots,m-1$, map every bounded subset of $\Hi$ to a bounded set,
	\end{itemize}
	then the sequence $\{ {x^n}\} _{n = 1}^\infty$ converges strongly to $\bar u$, the unique solution of the problem (\ref{vip-2}).
\end{corollary}
\begin{proof} Observe that the  convergence Theorem \ref{main-thm} is depended on the assumptions that the sequence $\{ F({x^n})\} _{n = 1}^\infty$ is bounded and  the operators $T_i, i=1,\ldots,m$, satisfy the DC principle. If we verify that these two mentioned assumptions are true, the convergence is a consequence of Theorem \ref{main-thm}. 
	
	Now, since the operator $F$ is Lipschitz continuous and the generated sequence $\{ {x^n}\} _{n = 1}^\infty$ is bounded, it follows that the sequence $\{ F({x^n})\} _{n = 1}^\infty$ is also bounded. On the other hand, it is noted from \cite[Theorem 4.2.7]{C12} that for a continuous convex function which is satisfying (i), we have that its corresponding subgradient projection will be satisfying the DC principle. Consequently, this means that  the operators $T_i, i=1,\ldots,m-1$,  are satisfying the DC principle. Moreover, we know from \cite[Proposition 16.20]{BC17} that for a continuous convex function, the assumptions (i) - (iii) are equivalent. This gives us that these three assumptions are the sufficient conditions for the fact that operators $T_i, i=1,\ldots,m-1$,  are satisfying the DC principle. Furthermore, since the metric projection $\mathrm{proj}_{C_m}$ is a nonexpansive operator (see, \cite[Theorem 2.2.21]{C12}), it follows that $T_m$ is also satisfying the DC principle. Hence, the assumptions of Theorem \ref{main-thm} are satisfied, and we therefore conclude that the sequence $\{ {x^n}\} _{n = 1}^\infty$ converges strongly to the unique solution of the problem (\ref{vip-2}) as desired.
	\qed
\end{proof}
\begin{remark}\begin{itemize}
		\item[(i)] It is very important to note that the continuity of $c_i, i=1,\ldots,m$, and the assumptions (i) - (iii) used in Corollary \ref{cor} can be dropped whenever the whole Hilbert space $\Hi$ is finite dimensional, see \cite[Corollary 8.40 and Proposition 16.20]{BC17} for further details. 
		\item[(ii)] An example of the simple closed convex and bounded  set $C_m$ in a general Hilbert space is nothing else than a closed ball $C_m:=\{x\in\Hi:\|x-z\|\leq r\}$, where $z\in\Hi$ is the center, and $r>0$ is the radius. In particular, if $\Hi=\R^n$, the finite-dimensional Euclidean space, an additional example is a box constraint $C_m:=[a_1,b_1]\times[a_2,b_2]\times\cdots\times[a_n,b_n]$, where $a_i,b_i\in\R$ with $a_i\leq b_i, i=1,\ldots,n$. For the closed-form formulae of these simple sets, the reader may consult \cite[Subsections 4.1.6 and 4.1.7]{C12}.
	\end{itemize}
\end{remark}

\section{Numerical Result}

In this section we report the convergence of ESCoM-CGD by the minimum-norm problem to a system of homogeneous linear inequalities with box constraint.  Suppose that we are given a  matrix $\mathbf{A}=[\mathbf{a}_1|\cdots|\mathbf{a}_m]^\top\in \R^{m\times k}$  of predictors  $\mathbf{a}_i=(a_{1i},\ldots,a_{ki})\in\R^k$,  for all $i=1\ldots,m$.  The approach of the considered problem with a box constraint is to find the vector $x\in\R^k$ that solves the problem 

\begin{eqnarray*}
	\begin{array}{ll}
		\textrm{minimize }\indent \frac{1}{2}\|x\|^2\\
		\textrm{subject to}\indent \mathbf{A}x\leq_{\R^m}\mathbf{0}_{\R^m},\\ 
		\indent\indent\indent\indent x\in [u,v]^k,
	\end{array}%
\end{eqnarray*}
or equivalently, in the explicit form,
\begin{eqnarray*}
	\begin{array}{ll}
		\textrm{minimize }\indent \frac{1}{2}\|x\|^2\\
		\textrm{subject to}\indent \<\mathbf{a}_i, x\>\leq 0, i=1\ldots,m, x\in [u,v]^k,
	\end{array}%
\end{eqnarray*}
where $u,v\in\R$ with  $u\leq v$.

Of course, this  minimum-norm problem can be written in the form of Problem \ref{Problem-VIP} as: finding $x^*\in \bigcap\limits_{i = 1}^{m+1} \fix{(\mathrm{proj}_{C_i})}$ such that 
\[\langle {x^*},x - {x^*}\rangle  \ge 0 \indent\text{    for all $x\in \bigcap\limits_{i = 1}^{m+1} \fix{(\mathrm{proj}_{C_i})}$.   }\]
where the constrained sets $C_i:=\{x\in\R^k:\<\mathbf{a}_i,x\>\leq 0\}, i=1,\ldots,m,$ are half-spaces  and  $C_{m+1}:=\{x\in\R^k:x\in [u,v]^k\}$ is a box constraint. It is clear that this variational inequality problem satisfies all assumptions of Problem \ref{Problem-VIP} by setting $F:=Id$, the identity operator, which is $1$-strongly monotone and $1$-Lipschitz continuous, and $T_i:=\mathrm{proj}_{C_i}, i=1,\ldots,m+1,$ the metric projections onto $C_i$ which are cutters with $\fix{(T_i)}=C_i$ and satisfying the demi-closed principle. Moreover, since the box constrained set $C_{m+1}$ is bounded, we have that the generated sequence $\{ {x^n}\} _{n = 1}^\infty$ is a bounded sequence which subsequently yields the boundedness of $\{ {F(x^n)}\} _{n = 1}^\infty$. This means that the assumptions of Theorem \ref{main-thm} are satisfying. All the experiments were performed under MATLAB 9.6 (R2019a) running on a MacBook Air 13-inch, Early 2015 with a 1.6GHz Intel Core i5 processor and 4GB 1600MHz DDR3 memory. All CPU times are given in seconds.

We generate the  matrix $\mathbf{A}$ in $\R^{m\times k}$ where $m=1000$ and $k=200$ by uniformly distributed random generating between $(-5,5)$ and choose the box constraint with boundaries $u=-1$ and $v=1$. The initial point is a vector whose all coordinates are normally distributed randomly chosen in $(0,1)$.  
In order to justify the advantages of the proposed Algoritgm \ref{algorithm}, we thus choose the hybrid conjugate gradient method (HCGM) \cite{IY09} and hybrid three term conjugate gradient method (HTCGM) \cite[Algorithm 6]{I11} as the benchmarks for the numerical comparisons.
In this situation, we set the operator considered in \cite{IY09,I11} by $T:=\mathrm{proj}_{C_{m+1}}\mathrm{proj}_{C_{m}}\cdots \mathrm{proj}_{C_1}$, which is a nonexpansive operator. Since the minimum-norm solution has the unique solution, in all following numerical experiments, we terminate the experimented methods when the norm become small, i.e., $\|x^n\|\leq 10^{-6}$.  We use 10 samplings for different randomly chosen matrix $\mathbf{A}$ and the initial point when performing each combination, and the presented results are averaged. We manually select the involved parameters of each compared algorithm and show some results when it achieves the fairly best performance.

Firstly, we demonstrate the effectiveness of step-size sequence $\varphi_n:=\frac{1}{(n+1)^a}$; where $a>0$, when ESCoM-CGD, HCGM, and HTCGM are applied for solving the above minimum-norm problem. We choose different values $a=0.005, 0.01, 0.05$, and $0.1$, and fix the corresponding paramter $\mu=1$, the step-size sequence $\beta_n=\frac{1}{(n+1)^{0.5}}$, and additionally set $\lambda_n=0.7$ for ESCoM-CGD. We plot the number of iterations and computational time in seconds with respect to different choices of $a$ in Figure \ref{varphi}.
\begin{figure}[h]
	\begin{center}
		\subfigure{
			\resizebox*{5.5cm}{!}{\includegraphics{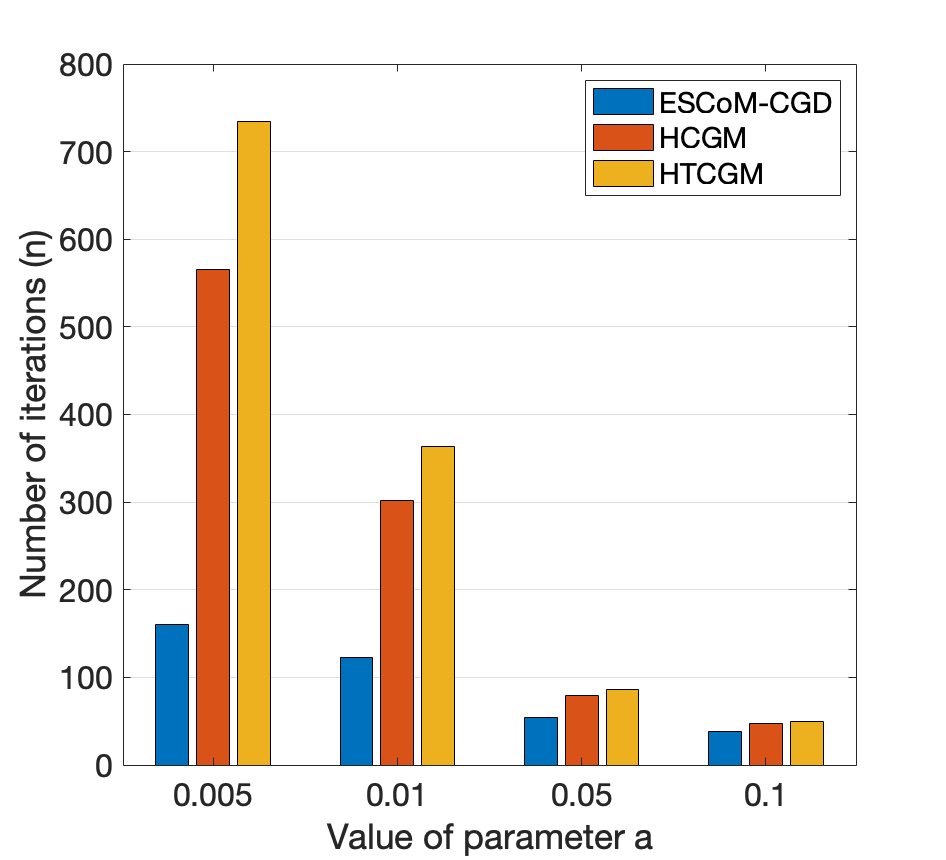}}}
		\subfigure{
			\resizebox*{5.5cm}{!}{\includegraphics{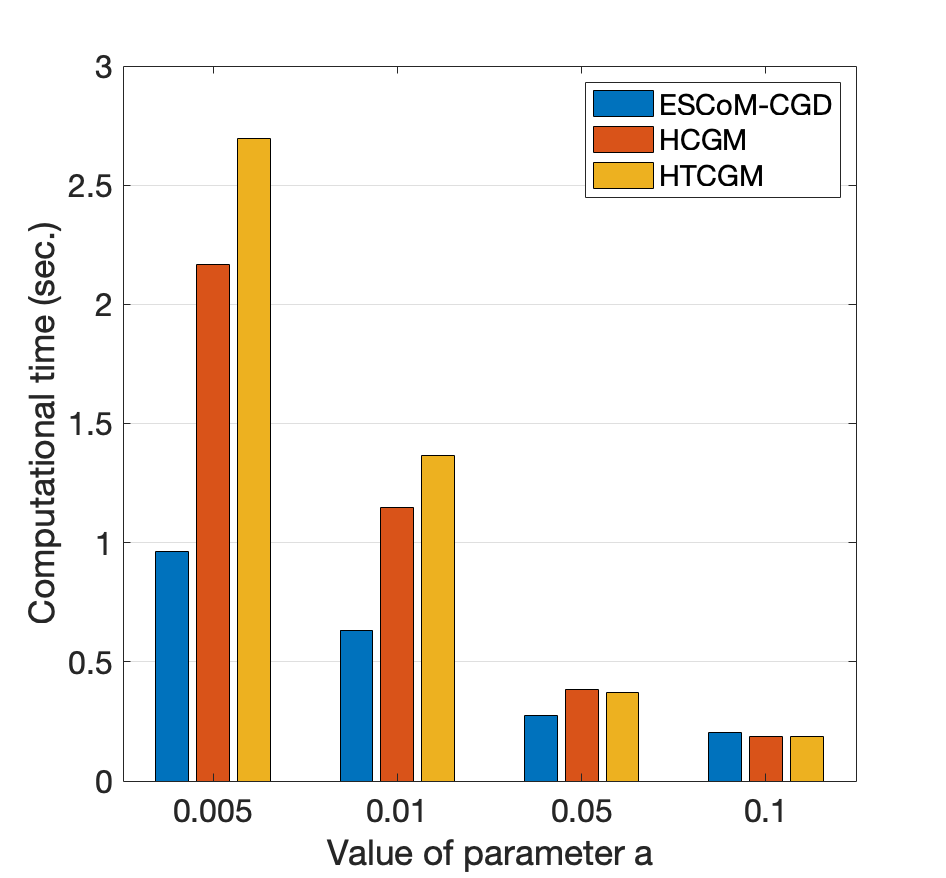}}}
		\caption{Influences of the step sizes $\varphi_n=1/(n+1)^{a}$ for several paramters  $a>0$ when performing ESCoM-CGD, HCGM \cite{IY09} and HTCGM \cite{I11}.}
		\label{varphi}
	\end{center}
\end{figure}

According to the plots in Figure \ref{varphi}, we see that the larger the value of $a$ yields the faster convergence in the senses of it need the smaller number of iterations and less computational time. We also see  that the proposed ESCoM-CGD is really faster than other methods, where the best result is observed for $a=0.1$. Notice that HCGM and HTCGM are very sensitive to the value of $a$, while ESCoM-CGD seems not. In fact, for $a=0.005$, HCGM and HTCGM require more than 550 ierations, whereas for $a=0.1$, it require approximately 50 iterations.

Next, we testify the influence of step-size sequence $\beta_n:=\frac{1}{(n+1)^b}$; where $0<b\leq1$, for the tested methods. We fix $\mu=1$, $\lambda_n=0.7$, and $\varphi_n=\frac{1}{(n+1)^{0.1}}$. We choose different value of $b$ in the interval $(0,1]$, namely, $b=0.01, 0.05, 0.1$, and $0.5$. The number of iterations and computational time in seconds for each choice of $b$ are plotted in Figure \ref{beta}.
\begin{figure}[h]
	\begin{center}
		\subfigure{
			\resizebox*{5.5cm}{!}{\includegraphics{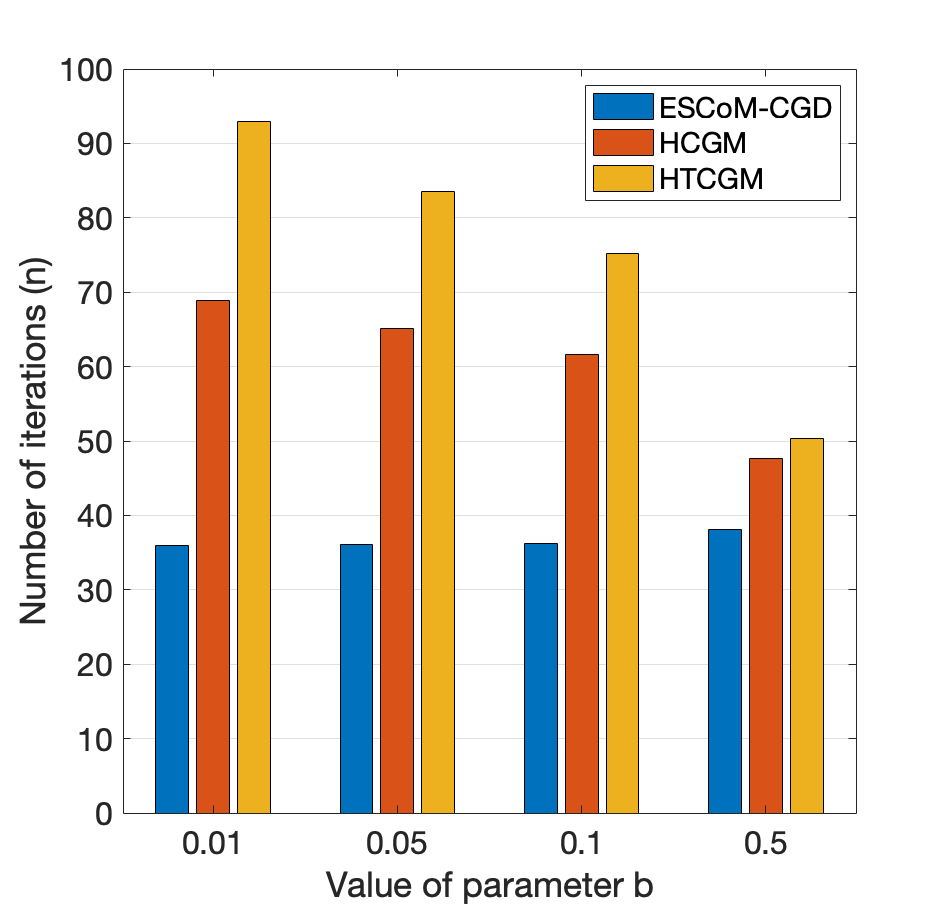}}}
		\subfigure{
			\resizebox*{5.5cm}{!}{\includegraphics{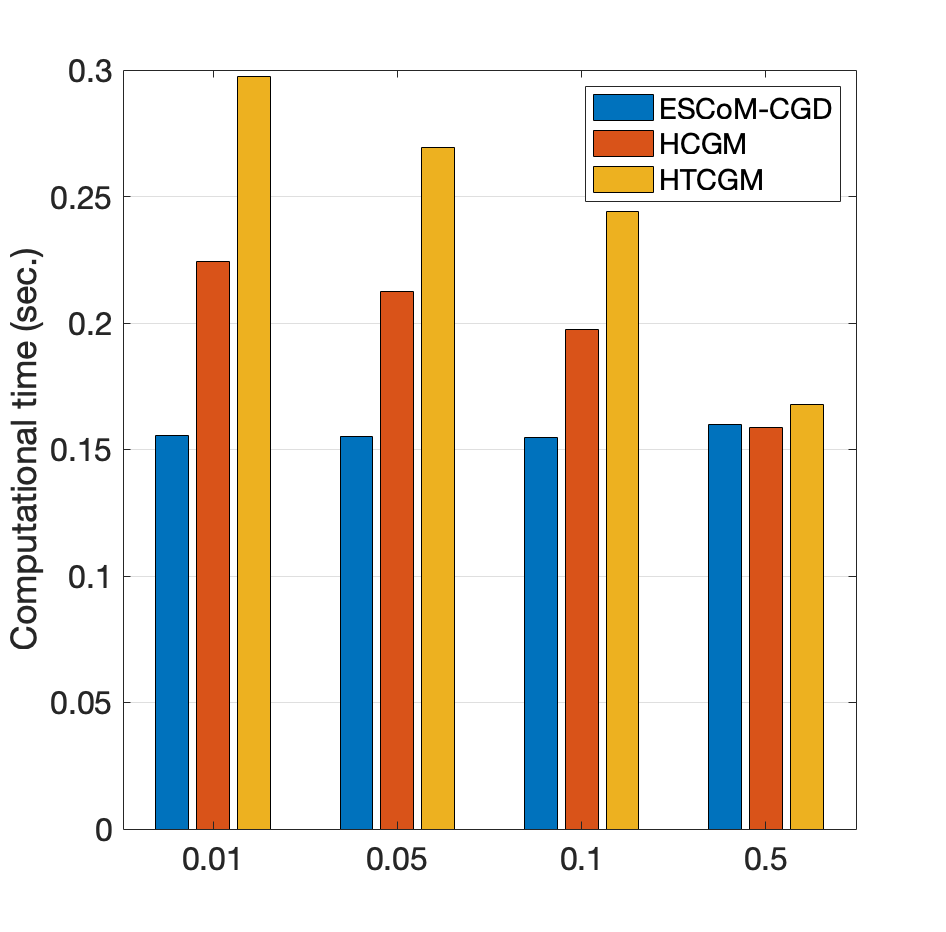}}}
		\caption{Influences of the step sizes $\beta_n=1/(n+1)^{b}$ for several paramters  $0<b\leq1$ when performing ESCoM-CGD, HCGM \cite{IY09} and HTCGM \cite{I11}.}
		\label{beta}
	\end{center}
\end{figure}

It can be seen from Figure \ref{beta} that ESCoM-CGD gives the best results for all values $b$. Moreover, their number of iterations and computational time seem indifferent to the different choices of $b$. For HCGM and HTCGM, we observe the the number of iterations as well as computational time decrease when the values $b$ grow up. For the exact results, the value $b=0.01$ is the best choice for ESCoM-CGD, however, the value $b=0.5$ is the best choice for both HCGM and HTCGM, which is coherent with the assertions in \cite{IY09,I11}.

In Figure \ref{mu}, we illustrate behaviour of the methods with respect to parameter $\mu\in(0,2)$. We fix $\varphi_n=\frac{1}{(n+1)^{0.1}}$ and $\lambda_n=0.7$. Moreover, we fix the best choices $\beta_n=\frac{1}{(n+1)^{0.01}}$ for ESCoM-CGD and $\beta_n=\frac{1}{(n+1)^{0.5}}$ for both HCGM and HTCGM. We choose different value $\mu=10^{-4}, 10^{-3}, 0.01, 0.1, 0.5, 1.0, 1.5$, and $1.9$. According to the plots, we observe that the very small value of $\mu=10^{-4}$ yields the best results for all methods. As a matter of fact, even if the best result for ESCoM-CGD is obtained for very small value $\mu$, we see that the method with large value $\mu>1$ also perform well. The overall best result is observed for HTCGM, this means that the assertion in \cite{I11} is confirmed again.

\begin{figure}[h]
	\begin{center}
		\subfigure{
			\resizebox*{5.5cm}{!}{\includegraphics{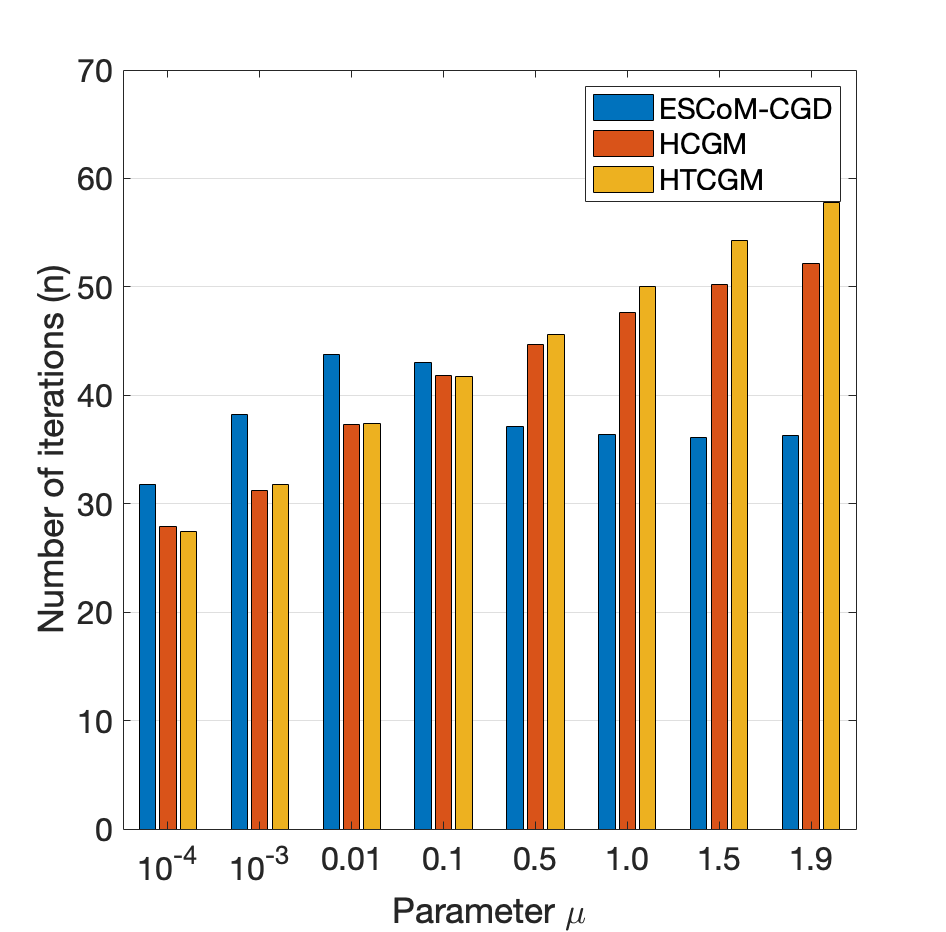}}}
		\subfigure{
			\resizebox*{5.5cm}{!}{\includegraphics{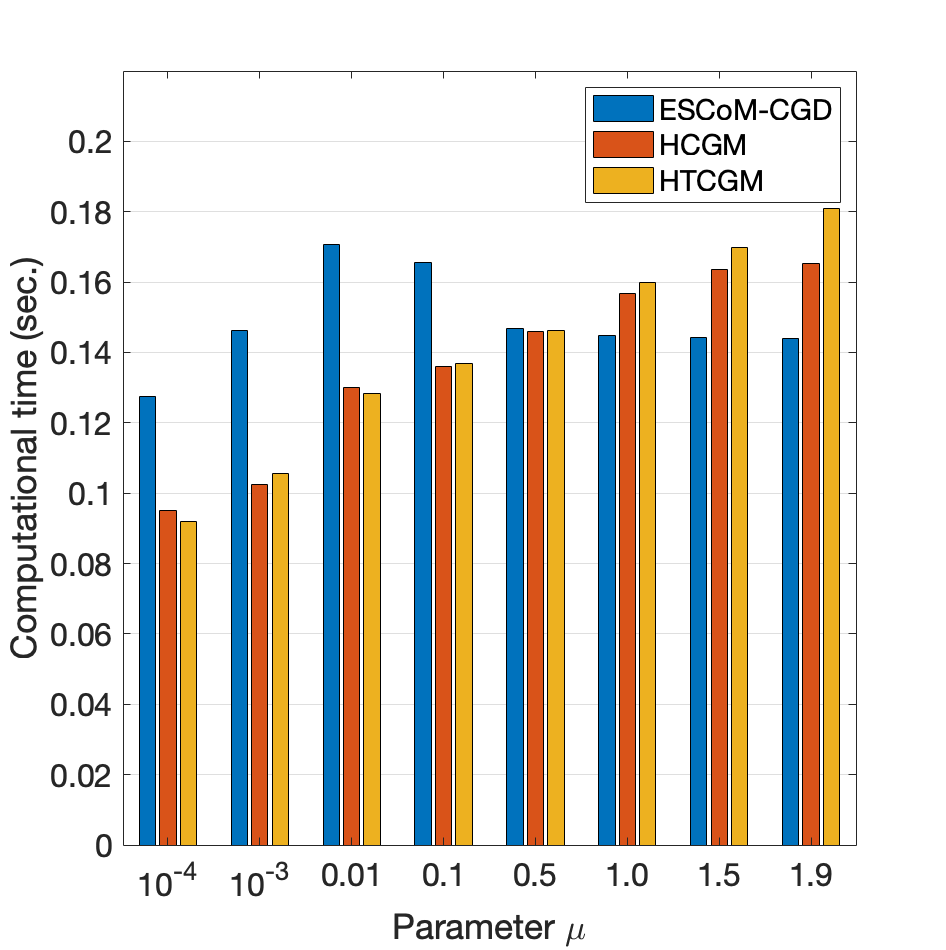}}}
		\caption{Influences of  parameter $\mu$ when performing ESCoM-CGD, HCGM \cite{IY09} and HTCGM \cite{I11}.}
		\label{mu}
	\end{center}
\end{figure}

As it is well-known the the presence of an appropriate relaxation parameter $\lambda_n\in(0,2)$ in MECSPM, or even the state-of-the-art relaxation methods can make the methods converge faster. Now, we demonstrate the influence of the relaxation paramter $\lambda_n$ when ESCoM-CGD is performed for solving the considered problem. We fix $\mu=10^{-4}, \varphi_n=\frac{1}{(n+1)^{0.1}}$, and $\beta_n=\frac{1}{(n+1)^{0.01}}$. We test a set of parameter $\lambda_n\in\{0.1,0.2,\ldots,1.9\}$, and plot the number of iterations and computational time with respect to different choices of $\lambda_n$ in Figure \ref{lambda}.

\begin{figure}[h]
	\begin{center}
		\subfigure{
			\resizebox*{5.5cm}{!}{\includegraphics{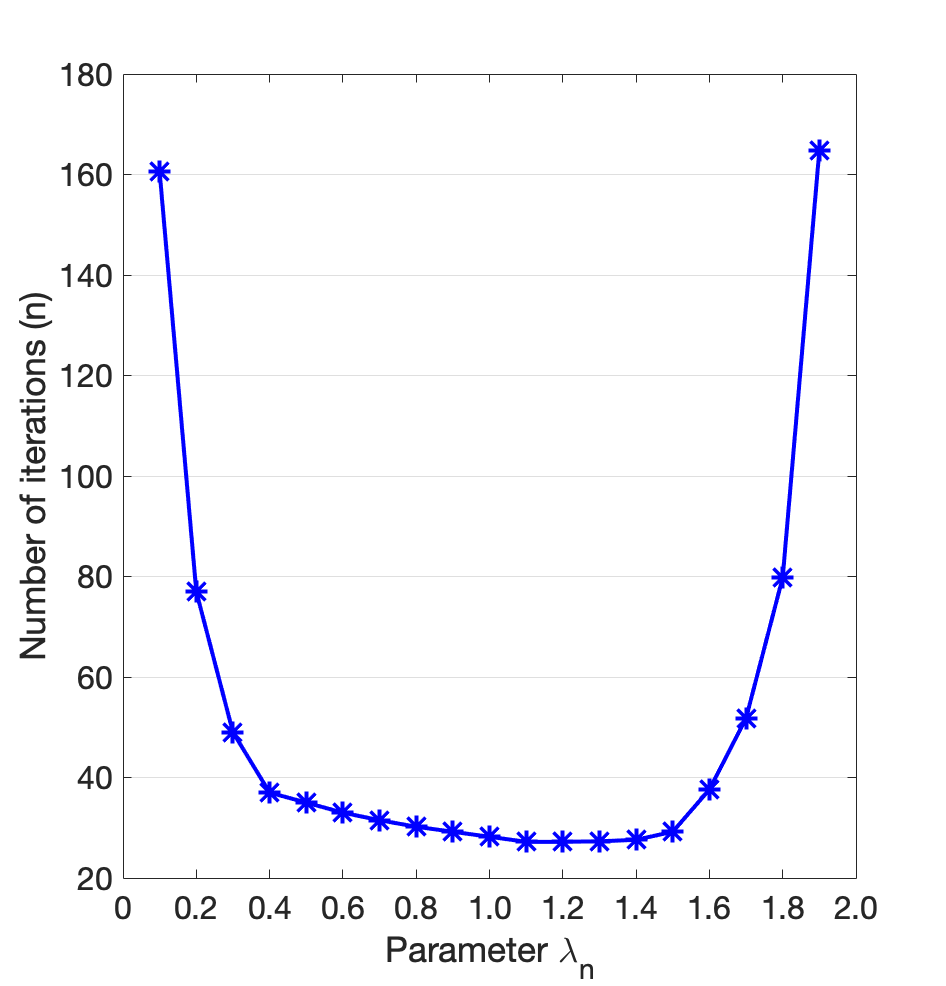}}}
		\subfigure{
			\resizebox*{5.5cm}{!}{\includegraphics{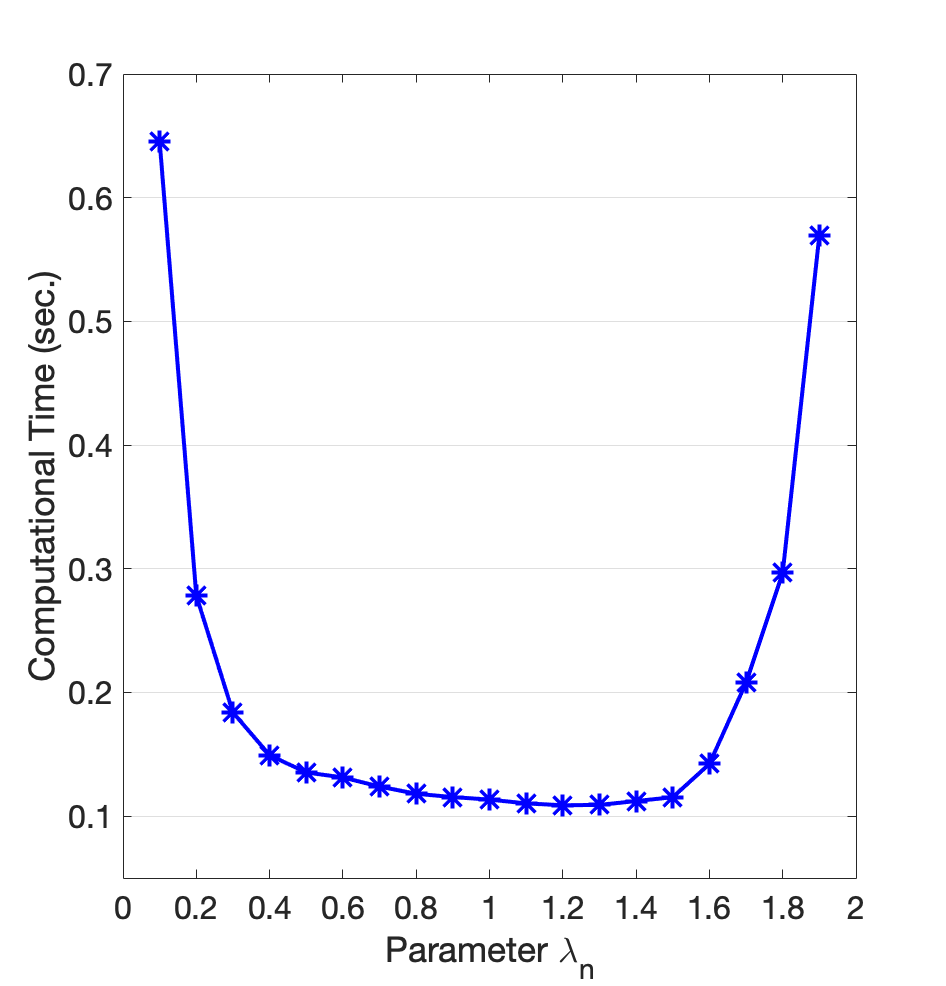}}}
		\caption{Influences of relaxation parameter $\lambda$ when performing ESCoM-CGD.}
		\label{lambda}
	\end{center}
\end{figure}

According to the curves in Figure \ref{lambda}, we see that the relaxation paramter $\lambda_n$ behaves significantly well convergence for a wide range of choices. In fact, we observe the the faster convergence is obtained for some intermediate choices of $\lambda_n\in[0.8,1.5]$, and the exactly best result is observed for $\lambda_n=1.2$. This observation relatively conforms to the numerical experiments in \cite{CN19}.

Finally, to showcase the superiority of our ESCoM-CGD, we compare the methods for various size $(m,k)$ of randomly matrix $\mathbf{A}$. We fix the corresponding parameters  as in Table \ref{para}. To show performance of the methods, the number of iterations with respect to the size of $\mathbf{A}$ are plotted in Figure \ref{various}. Moreover, we also present computational time in seconds with respect to the sizes $(m,k)$ in Table \ref{compare--tb}.

\begin{table}[h]
	\caption{Best choice of parameters used for performing ESCoM-CGD, HCGM \cite{IY09} and HTCGM \cite{I11}.}
	\label{para}       
	\begin{tabular}{l  c c c c c}
		\hline\noalign{\smallskip}
		Parameter& $\varphi_n$ &  $\beta_n$ & $\mu$ & $\lambda_n$  \\
		\noalign{\smallskip}\hline\noalign{\smallskip}
		ESCoM-CGD &$\frac{1}{(n+1)^{0.1}}$ & $\frac{1}{(n+1)^{0.01}}$ & $10^{-4}$ & 1.2\\
		HCGM \cite{IY09} &$\frac{1}{(n+1)^{0.1}}$ & $\frac{1}{(n+1)^{0.5}}$ & $10^{-4}$ & -\\
		HTCGM \cite{I11} &$\frac{1}{(n+1)^{0.1}}$ & $\frac{1}{(n+1)^{0.5}}$ & $10^{-4}$ & -\\
		\noalign{\smallskip}\hline
	\end{tabular}
\end{table}

\begin{figure}[h]
	\begin{center}		
		\resizebox*{12cm}{!}{\includegraphics{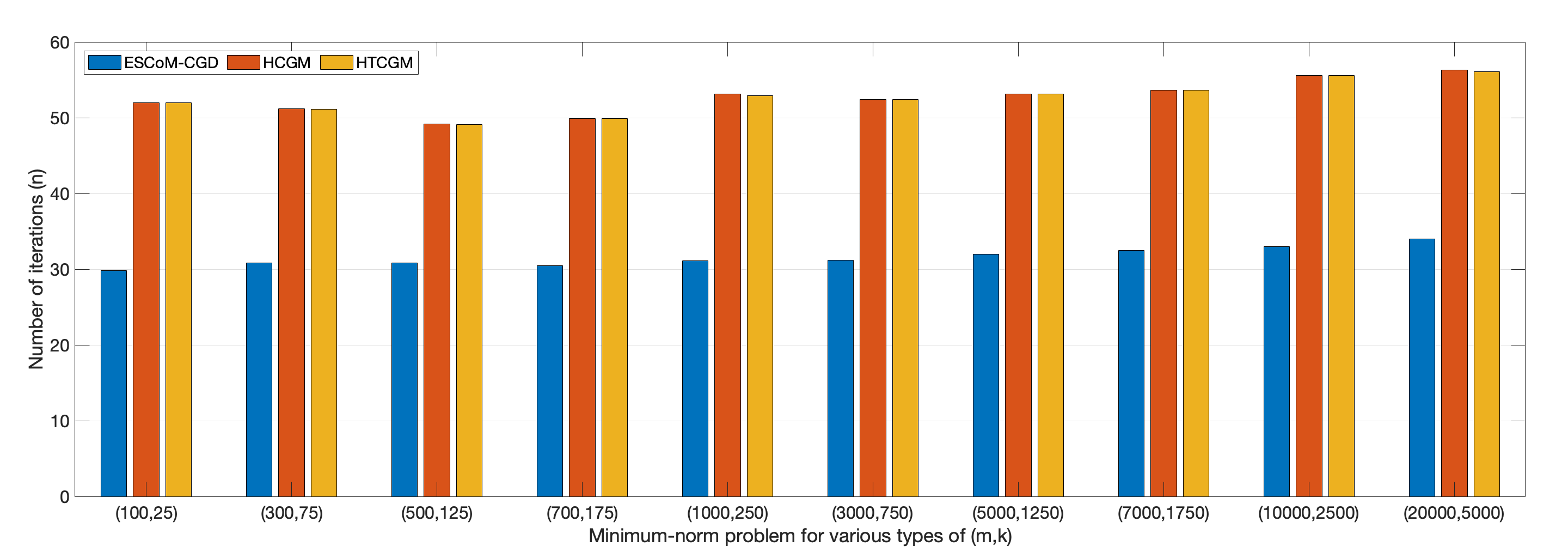}}
		\caption{Number of iterations when performing ESCoM-CGD, HCGM \cite{IY09} and HTCGM \cite{I11} for different choices of number of constraints ($m$) and dimensions ($k$).}
		\label{various}
	\end{center}
\end{figure}

\begin{table}[h]
	\caption{\label{compare--tb}Computational time in seconds when performing ESCoM-CGD, HCGM \cite{IY09} and HTCGM \cite{I11} for different choices of number of constraints ($m$) and dimensions ($k$).}       
	\begin{tabular}{l  r r  r r  r r  r r r  r }
		\hline\noalign{\smallskip}
		$(m,k)$& ESCoM-CGD &  HCGM \cite{IY09} & HTCGM \cite{I11}  \\
		\noalign{\smallskip}\hline\noalign{\smallskip}
			(100,25)&		0.0316&	0.0375&	0.0391\\
		(300,75)&	0.0621&	0.0757&	0.0699\\
		(500,125)&	0.1029&	0.1370&	0.1302\\
		(700,175)&	0.1499&		0.1874&	0.1758\\
		(1000,250)&	0.2045&		0.2717&	0.2557\\
		(3000,750)&	0.9660&	1.2123&	1.1694\\
		(5000,1250)&	2.2990&		3.0666&	 3.0382\\
		(7000,1750)&	6.9761&		9.2777&	9.2878\\
		(10000,2500)&	25.3301&	35.7751& 	35.4769\\
		(20000,5000)&	105.7223&	143.8273&	144.0225 \\
		\noalign{\smallskip}\hline
	\end{tabular}
\end{table}

The plots in Figure \ref{various} show that ESCoM-CGD  gives the best convergence results for all choices of $(m,k)$. Moreover, we see that HCGM  and HTCGM reach the optimal tolerance at most the same number of iterations.
Likewise, the results given in Table \ref{compare--tb} reveal that ESCoM-CGD reaches the optimal tolerance faster than both HCGM and HTCGM. It is worth noting that when the size $(20000,5000)$, ESCoM-CGD requires computational time less than other two methods approximately 40 seconds. This underlines the essential superiority of the proposed ESCoM-CGD.

\section{Conclusion}

The object of this work was the solving of a variational inequality problem governed by a strongly monotone and Lipschitz continuous operator over the intersection of fixed-point sets of cutter operators. We associated to it the so-called extrapolated sequential constraint method with conjugate gradient direction. We proved strong convergence of the generated sequence of iterates to the unique solution to the considered problem. Our numerical experiments show that the proposed method has a better convergence behaviour compared to other two methods.
For future work, one may consider and analyze a variant of the proposed method by using some constrained selections, e.g. the so-called dynamic string averaging procedure, for dealing with the constrained operators.

\begin{acknowledgements}
Mootta Prangprakhon is partially supported by  Science Achievement Scholarship of Thailand  (SAST), and Faculty of Science, Khon Kaen University.  Nimit Nimana is supported by Khon Kaen University.
\end{acknowledgements}

%
%



\end{document}